\documentclass[12pt,notitlepage]{amsart}
\usepackage{latexsym,amsfonts,amssymb,amsmath,amsthm}
\usepackage{color}
\usepackage[inner=1.0in,outer=1.0in,bottom=1.0in, top=1.0in]{geometry}

\newcommand{\mz}{\ensuremath{\mathbb Z}}
\newcommand{\mr}{\ensuremath{\mathbb R}}
\newcommand{\mh}{\ensuremath{\mathbb H}}

\newcommand{\mc}{\ensuremath{\mathbb C}}

\newcommand{\shortmod}{\ensuremath{\negthickspace \negthickspace \negthickspace \pmod}}

\newcommand{\half}{\ensuremath{ \frac{1}{2}}}

\newcommand{\intR}{\int_{-\infty}^{\infty}}

\newcommand{\sumstar}{\sideset{}{^*}\sum}

\newcommand{\SL}[2]{SL_{#1}(\mathbb{#2})}

\newcommand{\Sp}[2]{Sp_{#1}(\mathbb{#2})}

\DeclareMathOperator{\vol}{vol}

\theoremstyle{plain}		
	\newtheorem{mytheo}{Theorem} [section]
	
	\newtheorem{myprop}[mytheo]{Proposition}
	\newtheorem{mycoro}[mytheo]{Corollary}
     \newtheorem{mylemma}[mytheo]{Lemma}
	
	\newtheorem{myconj}[mytheo]{Conjecture}

\theoremstyle{remark}

\numberwithin{equation}{section}
\begin{document}

\author{Sheng-Chi Liu} 
\address{Department of Mathematics \\
	  Texas A\&M University \\
	  College Station \\
	  TX 77843-3368 \\
		U.S.A.}
\email{scliu@math.tamu.edu}

\author{Matthew P. Young} 
\address{Department of Mathematics \\
	  Texas A\&M University \\
	  College Station \\
	  TX 77843-3368 \\
		U.S.A.}
\email{myoung@math.tamu.edu}
\thanks{This material is based upon work supported by the National Science Foundation under agreement No. DMS-0758235.  Any opinions, findings and conclusions or recommendations expressed in this material are those of the authors and do not necessarily reflect the views of the National Science Foundation.}

\begin{abstract}
We study the analytic behavior of the restriction of a Siegel modular form to $\mh \times \mh$ in the case that the Siegel form is a Saito-Kurokawa lift.  A formula of Ichino links this behavior to a family of $GL_3 \times GL_2$ $L$-functions.
\end{abstract}
\title[Growth and nonvanishing of restricted Siegel modular forms]{Growth and nonvanishing of restricted Siegel modular forms arising as Saito-Kurokawa lifts}
\maketitle
\section{Introduction}
\subsection{The restriction of a Siegel modular form}
Suppose that $F: \mathcal{H}_2 \rightarrow \mc$ is a Siegel modular form of weight $k$ for $Sp_4(\mz)$, where Siegel's upper half space $\mathcal{H}_2$ is defined by
\begin{equation}
 \mathcal{H}_2 = \{ Z \in Mat_{2 \times 2}(\mc): Z = Z^{t} \text{ and Im}(Z) \text{ is positive definite} \}. 
\end{equation}
In coordinates, if $Z = \begin{pmatrix} \tau & z \\ z & \tau' \end{pmatrix}$ then $Z \in \mathcal{H}_2$ means $\tau, \tau' \in \mh$ and $\text{Im}(z)^2 < \text{Im}(\tau) \text{Im}(\tau')$.  Recall that $Sp_4(\mz)$ acts on $\mathcal{H}_2$ via $\gamma = \begin{pmatrix} A & B \\ C & D \end{pmatrix}$ with $\gamma(Z) = (AZ+B)(CZ+D)^{-1}$.  Then $F$ is a Siegel modular form of weight $k$ if $F(\gamma(Z)) = \det(CZ+D)^k F(Z)$ for all $\gamma \in Sp_4(\mz)$ and $Z \in \mathcal{H}_2$ and if $F$ is holomorphic in $\tau, z, \tau'$.  The Koecher principle means that such an $F$ automatically is bounded on a fundamental domain for $Sp_4(\mz) \backslash \mathcal{H}_2$.

In this paper we study the behavior of Siegel modular forms when restricted to $z=0$.  One can easily check that $SL_2(\mz) \times SL_2(\mz)$ embeds into $Sp_4(\mz)$ as follows.  Say $\alpha, \alpha' \in SL_2(\mz)$ with $\alpha = \begin{pmatrix} a & b \\ c & d \end{pmatrix}$, $\alpha' = \begin{pmatrix} a' & b' \\ c' & d '\end{pmatrix} $, then
\begin{equation}
 \gamma = \begin{pmatrix} a & & b & \\ & a' & & b' \\ c & & d & \\ & c' & & d'  \end{pmatrix} \in Sp_4(\mz);
\end{equation}
furthermore, for $Z = \begin{pmatrix} \tau & 0 \\ 0 & \tau' \end{pmatrix}$, 
we have $\gamma(Z) = \begin{pmatrix} \frac{a \tau + b}{c \tau + d} & 0 \\ 0 & \frac{a' \tau' + b'}{c' \tau' + d'} \end{pmatrix}$, and hence $F\begin{pmatrix} \tau & 0 \\ 0 & \tau' \end{pmatrix}$ is a modular form in $\tau$ and in $\tau'$.  It is an interesting question to understand how this restricted function $F|_{z=0}$ behaves.  In particular, we wish to understand when $F|_{z=0}$ vanishes identically, and also to compare the Petersson $L^2$ norm of $F|_{z=0}$ on $SL_2(\mz) \backslash \mh \times SL_2(\mz) \backslash \mh$ to the Petersson norm of $F$ on $Sp_4(\mz) \backslash \mathcal{H}_2$.  This measures in some sense how concentrated $F$ is along $F|_{z=0}$.

By linear algebra, one can write $F\begin{pmatrix} \tau & 0 \\ 0 & \tau' \end{pmatrix} = \sum_{g_1, g_2} c_{g_1, g_2} g_1(\tau) g_2(\tau')$, where $g_1$ and $g_2$ range over an orthonormal basis of weight $k$ modular forms and $c_{g_1, g_2}$ is the projection of $F$ onto $g_1 \times g_2$, with respect to the Petersson inner product.  The size and nonvanishing properties of $F|_{z=0}$ are then encoded by the behavior of these coefficients.  
A beautiful result of Ichino \cite{Ic}
gives a formula for $|c_{g, g}|^2$ as a central value of a $GL_3 \times GL_2$ $L$-function in the case that $F$ arises as a Saito-Kurokawa lift of a Hecke eigenform, and when $g$ is a Hecke eigenform.  Under these assumptions, it is not hard to show that $c_{g_1, g_2} = 0$ unless $g_1 = g_2$ (see \cite{Ic}, Lemma 1.1).  Ichino's work is what drew us to this problem.

\subsection{Saito-Kurokawa lifts}
\label{section:SK}
In order to state Ichino's formula we need to 
review some theory of Saito-Kurokawa lifts. We refer to \cite{EZ} for a comprehensive treatment. In brief summary, there is a chain of isomorphisms between various spaces of modular forms which when combined allows one to construct a Siegel cusp form $F_f$ of even weight $\ell$ from a given weight $2\ell-2$ cusp form $f$ for $SL_2(\mz)$.  The Shimura correspondence as progressed by Kohnen \cite{Kohnen} constructs from $f$ a weight $\ell-\half$ form $h_f$ for $\Gamma_0(4)$ lying in the Kohnen + space. 
If $h_f(\tau) = \sum_{n \geq 3} c(n) e(n \tau)$, then 
the Fourier expansion of $F_f$ can be given explicitly in terms of the $c(n)$'s.  Below we give a more elaborate summary of this correspondence and set notation; we also give the correspondence in the reverse order of above.

Let $\ell$ be an even positive integer.
Let $M_{\ell}(\Sp{4}{ \mz})$ (resp. $S_{\ell}(\Sp{4}{\mz})$) denote the space of Siegel modular forms (resp. cusp forms) of weight $\ell$. 
  Each $F \in M_{\ell}(\Sp{4}{ \mz})$ has the Fourier expansion
$$ F(Z)= \sum_{N \geq 0} A(N) e(\text{tr} NZ),
$$
where the summation runs over positive semi-definite half-integral $2 \times 2$ matrices.
If we write $Z=\begin{pmatrix} \tau & z \\z &\tau' \end{pmatrix}$ 
and
$N= \begin{pmatrix} n& r/2 \\ r/2 & m \end{pmatrix}$ 
with $n, r, m \in \mz, n \geq 0, m \geq 0, r^2 \le 4nm$, then the Fourier expansion of $F$ can be rewritten as
\begin{align}
\label{eq:Ffourier}
 F(\tau, z, \tau') &= \sum_{\substack{n \geq 0, m \geq 0, r \in \mz \\  4nm-r^2 \geq 0}}
                    A(n, r, m) e(n \tau + rz + m \tau') \\
                  &=: \sum_{m=0}^{\infty} \phi_m(\tau, z) e(m \tau'). \label{FJ}
\end{align}
Here $\phi_m$ is a Jacobi form of weight $\ell$ and index $m$ (see \cite{EZ} for these definitions). The expansion (\ref{FJ}) is called the Fourier-Jacobi expansion.

Let $J_{\ell, m}$ denote the space of Jacobi forms of weight $\ell$ and index $m$. Let $M_{\ell}^*(\Sp{4}{ \mz})$ denote the Maass space which is the subspace of $M_{\ell}(\Sp{4}{ \mz})$ with Fourier coefficients satisfying
\begin{align}
 A(n,r,m)= \sum_{d|(n, r, m)}d^{\ell-1} A\left(\frac{nm}{d^2}, \frac{r}{d},1 \right) && (\forall n \geq 0, m \geq 0, r \in \mz).
\end{align}
\begin{mytheo} [\cite{EZ} Theorems 6.1, 6.2, 6.3] \label{Ma}
 The map $F \to \phi_1$ gives a Hecke isomorphism between $M_{\ell}^*(\Sp{4} {\mz})$ and $J_{\ell, 1}$.
Hence $F \in M_{\ell}^*(\Sp{4}{ \mz})$ is completely determined by its first Fourier-Jacobi coefficient $\phi_1.$
\end{mytheo}

On the other hand, there is an isomorphism between Jacobi forms and modular forms of half-integral weight.
Let $M_{\ell-1/2}^+(\Gamma_0(4))$ denote the Kohnen ``$+$'' space which is the subspace of $M_{\ell-1/2}(\Gamma_0(4))$ consisting of modular forms whose $n$-th Fourier coefficients vanish unless $(-1)^{\ell-1}n \equiv 0,1 \pmod {4}$.
\begin{mytheo} [\cite{EZ} Theorem 5.4] \label{Ja}
 The map
\begin{align} \label{Ja1}
 h(\tau):= \sum_{\substack{n \geq 0 \\ n \equiv 0, 3 \shortmod{4}}} c(n) e(n\tau) \to 
\phi(\tau, z):= \sum_{r^2 \le 4n} c(4n-r^2) e(n\tau + rz)
\end{align}
gives a Hecke isomorphism between  $M_{\ell-1/2}^+(\Gamma_0(4))$ and $J_{\ell, 1}$.
\end{mytheo}

Given $f \in M_{2\ell-2}(SL_2(\mz))$ a holomorphic modular form of weight $2\ell-2$, then using the Shimura correspondence, Theorem \ref{Ja}, and Theorem \ref{Ma}, one associates to it a Siegel modular form $F_f \in M_{\ell}^*(Sp(4, \mz))$.  Moreover if $f$ is a cusp form, then $F_f$ is also a cusp form.
We call $F_f$  the Saito-Kurokawa lift of $f$.

\subsection{Explicit $L^2$ norm formula}
\label{section:L2}
In this section, we will state Ichino's formula as well as some consequences of it.

Let $k$ be an odd positive integer. Let $f \in S_{2k}(\SL{2}{\mz})$ be a normalized (first Fourier coefficient equals 1) Hecke eigenform and $h \in S_{k + 1/2}^+ (\Gamma_0(4))$ a Hecke eigenform associated to $f$ by the Shimura correspondence. Let $F_f \in S_{k+1}(\Sp{4} {\mz})$ be the Saito-Kurokawa lift of $f$. 

Let $B_{k+1}$ denote a Hecke basis of $S_{k+1}(\SL{2}{ \mz})$, normalized so that the first Fourier coefficient equals $1$.
For each $g \in B_{k+1}$, the period integral $\left\langle F_f|_{z=0 }, g \times g  \right\rangle$ is given by 
$$\left\langle F_f|_{z=0 }, g \times g  \right\rangle =
\int_{\SL{2}{\mz}\backslash \mh} \int_{\SL{2}{\mz}\backslash \mh}
         F \left( \begin{pmatrix} \tau & \\ & \tau' \end{pmatrix}\right)  
     \overline{g(\tau) g(\tau')} \text{Im}(\tau)^{k+1} \text{Im}(\tau')^{k+1} d \mu(\tau) d\mu(\tau'),
$$
where $d \mu(z) = y^{-2} dx dy$ if $z = x + iy$.
We recall that the Petersson norms of $f, g, h$ are defined by
\begin{gather*}
 \left\langle f,f  \right\rangle  = \int_{\SL{2}{\mz}\backslash \mh} |f(z)|^2 y^{2k} d\mu(z),  
 \qquad
  \left\langle g,g  \right\rangle  = \int_{\SL{2}{\mz}\backslash \mh} |g(z)|^2 y^{k+1}d\mu(z), \\ 
  \left\langle h,h  \right\rangle  = \frac{1}{6}\int_{\Gamma_0(4)\backslash \mh} |h(z)|^2 y^{k+1/2} d\mu(z).
\end{gather*}
We also recall
\begin{equation}
 \langle F, F \rangle = \int_{Sp_4(\mz) \backslash \mathcal{H}_2} |F(Z)|^2 \det(Y)^{k+1} d\mu(Z),
\end{equation}
where $Y = \text{Im}(Z)$ and $d\mu \begin{pmatrix} u + iv & x + iy \\ x+iy & u'+iv' \end{pmatrix} = \frac{du dx du' dv dy dv'}{(vv'-y^2)^3}$ is an invariant measure for $Sp_4(\mz)$.  It is convenient to recall that with the above measures that
\begin{equation}
\label{eq:volumes}
 v_1:=\vol(SL_2(\mz) \backslash \mh) = 2 \pi^{-1} \zeta(2) = \frac{\pi}{3}, \qquad v_2:=\vol(Sp_4(\mz) \backslash \mathcal{H}_2) = 2 \pi^{-3} \zeta(2) \zeta(4) = \frac{\pi^3}{270}.
\end{equation}

Let $\lambda_f(n)$ (resp. $\lambda_g(n)$) denote the $n$-th Hecke eigenvalue of $f$ (resp. $g$), scaled so Deligne's bound gives $|\lambda_f(n)| \le d(n)$. The $L$-function associated to $f$ is given by
$$ L(s,f)= \sum_{n=1}^{\infty} \frac{\lambda_f(n)}{n^s},
$$
and its completed $L$-function satisfies a functional equation under $s \to 1-s$. 
We will normalize  all $L$-functions to have functional equations under $s \to 1-s$ in this paper.

Let $G$ be the cusp form on $GL_3$ which is the symmetric-square lift (Gelbart-Jacquet lift) of $g$ \cite{GJ}, with Fourier coefficients 
$A_{G}(m_1,m_2)$ satisfying
\begin{equation} \label{FSym}
 A_{G}(m_1, m_2)= \sum_{d| (m_1, m_2)} \mu (d) A_{G} \left(\frac{m_1}{d}, 1 \right) 
  A_{G} \left(\frac{m_2}{d}, 1 \right)
\end{equation}
and
\begin{equation}
\label{eq:FSym2}
A_{G}(r,1)= \sum_{ab^2=r} \lambda_g (a^2).
\end{equation}
The $GL_3 \times GL_2$ Rankin-Selberg $L$-function is defined by
\begin{equation}
\label{eq:RSdef}
L(s, \mathrm{sym}^2 g \otimes f)= \sum_{m_1=1}^{\infty}\sum_{m_2=1}^{\infty} \frac{ \lambda_f (m_1) A_{G}(m_1, m_2)}{(m_1 m_2^2)^s}.
\end{equation}
It is entire and satisfies the functional equation \cite{CP} \cite{JPS1} \cite{JPS2}
$$ \Lambda (s, \mathrm{sym}^2 g \otimes f )= \Lambda (1-s, \mathrm{sym}^2 g \otimes f )
$$
where
\begin{multline*}
 \Lambda (s, \mathrm{sym}^2 g \otimes f ) = 2^3 (2\pi)^{-(3s+3k-\tfrac{1}{2})} 
\Gamma \left(s+ 2k-\tfrac{1}{2}\right) \Gamma \left(s+k-\tfrac{1}{2} \right)  
   \Gamma \left(s +\tfrac{1}{2} \right)\\
 \times L(s, \mathrm{sym}^2 g \otimes f). 
\end{multline*}

Now we can state Ichino's result.
\begin{mytheo}[\cite{Ic} Theorem 2.1] \label{I2}
 For each $g \in B_{k+1}$,
\begin{equation}
\Lambda \left( \tfrac{1}{2}, \mathrm{sym}^2 g \otimes f  \right)
= 2^{k+1} \frac{\left\langle f, f \right\rangle }{\left\langle h, h  \right\rangle } 
  \frac{\left| \left\langle F|_{z=0 }, g \times g  \right\rangle \right|^2}
{\left\langle g, g  \right\rangle^2}.
\end{equation}
\end{mytheo}
Note that Ichino's formula shows the central values of the $L$-functions are nonnegative.  For $F$ and $g$ fixed Ichino's result in some sense gives a complete description of the (modulus squared) projection of $F|_{z=0}$ onto $g \times g$.  It is interesting to ask about the behavior of $F|_{z=0}$ which leads to understanding the average behavior over $g$ in which case one can hope for more refined information.  To this end, we make the following definition:
\begin{equation}
\label{eq:Normdef}
 N(F_f) = 
\frac{\frac{1}{v_1^2} \langle F_f|_{z=0}, F_f|_{z=0} \rangle}{\frac{1}{v_2} \langle F_f, F_f \rangle},
\end{equation}
where the inner product in the numerator is the product of Petersson inner products on $SL_2(\mz) \backslash \mh \times SL_2(\mz) \backslash \mh$.  
One can see from the Fourier expansion \eqref{eq:Ffourier} that $F_f|_{z=0}$ is a cusp form so that the inner product in the numerator is finite.
Clearly, $F_f|_{z=0}$ is identically zero if and only if $N(F_f) = 0$.
We divide by the volumes so that in practice we are comparing probability measures; we thank Akshay Venkatesh for this excellent suggestion. 

Using Ichino's formula as well as some nice simplifications we derive the following:
\begin{mytheo} \label{NF}
With the same notation as in this section, we have
\begin{equation}
\label{eq:NF}
  N(F_f) = \frac{v_1^{-2}}{v_2^{-1}} \frac{24 \pi }{L(3/2, f) L(1, \mathrm{sym}^2f)} 
   \sum_{g \in B_{k+1}} \frac{1}{k} L\left( \tfrac{1}{2}, \mathrm{sym}^2 g \otimes f  \right).
\end{equation}
\end{mytheo}
This derivation occurs at the end of this subsection.

It is pleasant to consider some small weight examples.  For $k+1=10$ there exists a cusp form $F_f$, commonly denoted $\chi_{10}$ \cite{Kurokawa} which vanishes along $z=0$.  This can be seen directly by its Fourier expansion but is also obvious from the right hand side of \eqref{eq:NF} since the sum over $g$ is empty.  For the same reason, when $k+1=14$ there is a form $\chi_{14}$ in the Maass space, but there are no $g$'s and $N(\chi_{14}) = 0$.  For other small values of $k$ we numerically computed $N(F_f)$ obtaining Table \ref{table:norms} where the labels ``a'' and ``b'' distinguish two forms of the particular weight.  
\begin{table}
\label{table:norms}
\caption{The norm for small weights}
\begin{tabular}{c|ccccccc}
$ k $ & $12$ & $16a$ & $16b$ & $18a$ & $18b$ & $20a$ & $20b$ \\
\hline
$\approx N(F_f) $ & $.83$ & $.64$ & $.49$ & $.043$ & $1.2$ & $.88$ & $.44$
\end{tabular}
\end{table}
It would be good to extend the table to larger weights, and to compute the norm more accurately; the second displayed digit may not be correct though it seems the calculation has stabilized enough to detect the first digit.
Based on these calculations as well as other heuristics (see Conjecture \ref{conj:N2} and Theorem \ref{thm:twistnonvanish} below) we conjecture that $N(F_f) \neq 0$ for $k+1 \geq 16$.  It would be interesting (and surprising) to find $f$ and $g$ such that $L(1/2, \mathrm{sym}^2 g \otimes f) = 0$.

Using the convexity bound $L\left( \frac{1}{2}, \mathrm{sym}^2 g \otimes f  \right) \ll k$, the well-known bound \cite{HL}
\begin{equation} \label{sym}
 (\log{k})^{-1} \ll L(1, \mathrm{sym}^2f)  
\end{equation}
and the ``trivial'' bound
\begin{align} \label{invf32}
 L \left(\tfrac{3}{2}, f \right)^{-1}  \le \prod_{p} (1+ p^{-3/2})^2,
\end{align}
we deduce the following corollary:
\begin{mycoro} 
\label{coro:convexitybound}
With the same notation as in this section, we have
 $$ N(F_f) \ll k \log{k}.
$$
\end{mycoro}
The Lindel\"{o}f hypothesis implies $N(F_f) \ll k^{\varepsilon}$ but in fact we have a much more refined conjecture:
\begin{myconj}
\label{conj:N2}
 As $k \to \infty$, we have
$$N(F_f) \sim 2. $$
\end{myconj}
This is a special case of the conjectures of \cite{CFKRS}, though it takes a little work to derive this particular answer from their conjecture, which we discuss in Section \ref{section:conjecture}.
It is a challenging problem to unconditionally prove Conjecture \ref{conj:N2}; to help measure the diffculty, one can consider the mean-value problem with $f$ replaced by an Eisenstein series (which past experience indicates should be an easier problem) leading to the problem of estimating the second moment of symmetric-square $L$-functions in the weight aspect, which is listed as an open problem in \cite[Conjecture 1.2]{Khan}.  We can prove Conjecture \ref{conj:N2} on average over $f$ and the weight $k$ (of course $f$ depends on the weight so we cannot try to fix $f$ and vary the weight):

\begin{mytheo} 
\label{thm:averageweight}
Suppose that $w$ is a function satisfying
\begin{equation}
\label{eq:wprop}
\begin{cases}
w \text{ is smooth with compact support on $[K, 2K]$} \\
|w^{(j)}(x)| \leq C_j K^{-j} \text{ for some $C_j > 0$}, \quad j=0,1,2, \dots.
\end{cases}
\end{equation}
Then as $K \rightarrow \infty$, 
\begin{equation}
   \sum_{k \ odd } w(k) \sum_{f \in B_{2k}} N(F_f) \sim  \sum_{k \ odd } w(k) \sum_{f \in B_{2k}} 2.
\end{equation}
\end{mytheo}
We sketch a proof of Theorem \ref{thm:averageweight} in Section \ref{section:sketch} while the full proof takes up the majority of the paper.

\begin{mycoro}
\label{coro:upperlower}
 The number of pairs $(f,k)$ with $f \in B_{2k}$ and $k$ odd with $K < k \leq 2K$ and $N(F_f) > A$ is at most $O(K^2/A)$ (meaning ``almost all'' $f$'s have $N(F_f) \leq k^{\varepsilon}$).  Furthermore, the number of pairs $(f,k)$ with $f \in B_{2k}$ and $k$ odd with $K < k \leq 2K$ and $N(F_f) > 1$ (in particular, not close to zero) is $\gg K/\log{K}$.
\end{mycoro}
\begin{proof}[Proof of Corollary \ref{coro:upperlower}]
 For the first statement, let $Q$ denote the number of such pairs $(f,k)$ with $N(F_f) > A$.  Then
\begin{equation}
 A Q \leq \sum_{\substack{K < k \leq 2K  \\ k \text { odd }}} \sum_{f \in B_{2k}} N(F_f),
\end{equation}
which is $\ll K^2$ by Theorem \ref{thm:averageweight} (after a smoothing argument).  For the second statement, we use Cauchy's inequality as follows
\begin{equation}
\label{eq:normcomparison}
 \sum_{\substack{K < k \leq 2K  \\ k \text { odd }}} \sum_{\substack{f \in B_{2k} \\ N(F_f) > 1}} N(F_f) \leq  (\sum_{\substack{K < k \leq 2K  \\ k \text { odd }}} \sum_{\substack{f \in B_{2k} \\ N(F_f) > 1}} 1)^{1/2} (\sum_{\substack{K < k \leq 2K  \\ k \text { odd }}} \sum_{\substack{f \in B_{2k} \\ N(F_f) > 1}} N(F_f)^2)^{1/2}.
\end{equation}
Since $N(F_f)$ is $2$ on average, the left hand side of \eqref{eq:normcomparison} is $\gg K^2$.  Let $Q'$ be the number of pairs $(f,k)$ with $N(F_f) > 1$.  Using the convexity bound (Corollary \ref{coro:convexitybound}), we see that the right hand side of \eqref{eq:normcomparison} is then 
\begin{equation}
 \ll Q'^{1/2} (K \log{K})^{1/2} (\sum_{\substack{K < k \leq 2K  \\ k \text { odd }}} \sum_{\substack{f \in B_{2k}}} N(F_f))^{1/2}.
\end{equation}
Using Theorem \ref{thm:averageweight} again, we find the stated lower bound on $Q'$.
\end{proof}

\begin{proof}[Proof of Theorem \ref{NF}.]
Here we show how to derive Theorem \ref{NF} from Theorem \ref{I2}. 
 By \cite[Lemma 1.1 and Remark 2.3]{Ic}, we have the following expansion
\begin{align} \label{I1}
  F \left( \begin{pmatrix} \tau & \\ & \tau' \end{pmatrix} \right) 
 = \sum_{g \in B_{k+1}} 
   \left\langle F|_{z=0}, \frac{g}{||g||} \times \frac{g}{||g||}  \right\rangle 
  \frac{g(\tau)}{||g||} \frac{g(\tau')}{||g||}.
\end{align}
Combining (\ref{I1}) and Theorem \ref{I2}, we have the following $L^2$ norm formula
\begin{align} \label{NF1}
  N(F_f)= \frac{v_1^{-2}}{v_2^{-1}} \sum_{g \in B_{k+1}} 2^{-(k+1)} \frac{\left\langle h, h \right\rangle}{\left\langle F, F \right\rangle}
      \frac{1}{\left\langle f, f \right\rangle} \Lambda \left( \tfrac{1}{2}, \mathrm{sym}^2 g \otimes f  \right).
\end{align}
By \cite[p.251]{Iw}, we have
$$L(1, \mathrm{sym}^2f)= \frac{\pi}{2} \frac{(4\pi)^{2k}}{\Gamma(2k)} \left\langle f, f \right\rangle.
$$ 
Furthermore, \cite[Corollary of Theorem 2]{KS} and \cite[Lemma 5.2]{Br} show (be aware that Brown's normalization of $\langle h, h \rangle$ differs from ours by our extra $1/6$ arising from the index of $\Gamma_0(4)$ in $\SL{2}{\mz}$)
 $$ \frac{\left\langle h, h \right\rangle }{\left\langle F, F \right\rangle} = \frac{2^3 3 \pi^{k+1}}{k!}\frac{ 1 }{ L(\tfrac{3}{2}, f)}.$$
 Combining the formulas gives the result.
\end{proof} 

\subsection{Nonvanishing of the $L^2$ norm}
In this section, we fix the weight $\ell$ (even) and study the nonvanishing of the restricted $L^2$ norm in two ways. 
One is using the structure of the theory of Jacobi forms while the other is from a lower bound on twisted $L$-functions.  Be aware that the weight $\ell$ used in this section and Section \ref{section:SK} corresponds to weight $k+1$ used in Section \ref{section:L2} and the rest of the paper.
\begin{mytheo} \label{NV1}
 Of the $\mathrm{dim}(S_{2\ell-2})= \frac{\ell}{6}+ O(1)$ Hecke eigenforms lifting to Siegel modular forms of weight $\ell$ under the Saito-Kurokawa correspondence, no more than $\mathrm{dim}(M_{\ell-10})= \frac{\ell}{12}+ O(1)$ have vanishing restricted $L^2$ norm $N(F_f)$.  Consequently, for such $f$ we have that there exists a Hecke eigenform $g$ of weight $\ell$ such that $L(1/2, \mathrm{sym}^2 g \otimes f) \neq 0$.
\end{mytheo}
This result says that at least $(\frac{1}{2}+ O(\frac{1}{\ell})) \cdot \mathrm{dim}(S_{2\ell-2})$ of the forms do not vanish. 
The proof of this theorem is based on the theory of Jacobi forms and does not directly show nonvanishing of any $GL_3 \times GL_2$ $L$-functions.  By the way, this nonvanishing result for $N(F_f)$ does not lead to a lower bound on $N(F_f)$ and only shows that for at least one $g$ that $L(1/2, \mathrm{sym}^2 g \otimes f) \neq 0$ by way of Ichino's formula.  Corollary \ref{coro:upperlower} complements this by showing that the norm is at least $1$ for at least $\gg K/\log{K}$ out of the $\asymp K^2$ forms $F_f$ under consideration.  Our results do not rule out the following ``worst-case'' behavior: for each weight $2\ell-2$, $50 \%$ of the forms have vanishing norm and the other $50 \%$ have nonvanishing norm, yet amongst the forms with nonvanishing norm, there exists one form $f$ with $N(F_f) \approx K$, and the rest with very small (say $\leq \exp(-\exp(K))$) but nonzero norm.  It would be interesting to obtain a better upper bound on $N(F_f)^2$ on average over $f$ and $k$ as one could improve the lower bound in Corollary \ref{coro:upperlower}.


We also present a different argument giving the following result.

\begin{mytheo} \label{NV2}
\label{thm:twistnonvanish}
 Suppose that for $\ell$ even large enough and for given $f$ a Hecke eigenform of weight $2\ell-2$, there exists a fundamental discriminant $D < 0$, $4|D$, such that $|D| \ll \ell^{1-\varepsilon}$ and satisfying $L(\frac{1}{2}, f \otimes \chi_D) \geq \ell^{-100}.$ Then $N(F_f) \neq 0$ for such $f$.
\end{mytheo}

The result of Iwaniec and Sarnak [IS] shows that the hypothesis holds with $D=-4$, say, for at least $(\half - \varepsilon) \cdot \dim(S_{2\ell-2}))$ of the forms (for $\ell$ large enough in terms of $\varepsilon$), on average over $\ell$; note that Theorem \ref{NV1} holds for individual $\ell$.
Recall that any improvement in Iwaniec-Sarnak's work to $ (\frac{1}{2}+ \varepsilon)$ would effectively remove the Landau-Siegel zero, however this requires a nonvanishing result uniformly in $D$ over a large range whereas we only require a single $D$, e.g., $D=-4$.
 
\begin{proof}[Proof of Theorem \ref{NV1}.]
For this proof we use the Fourier-Jacobi expansion of Siegel modular forms and properties of Jacobi forms.
There is a certain family of Hecke operators $V_q: J_{\ell,m} \to J_{\ell,mq}$ defined as follows:
\begin{equation} \label{Vl}
  (\phi|_{\ell,m}V_q)(\tau, z)= q^{\ell-1} \sum_{ad=q} \sum_{b \shortmod{d}} d^{-\ell} \phi(\frac{a\tau +b}{d}, az).
\end{equation}
Then the inverse map of the map in Theorem \ref{Ma} is given as follows: if $\phi \in J_{\ell,1}$, then
$$ F \begin{pmatrix}\tau & z \\ z & \tau' \end{pmatrix} =v(\phi):= \sum_{m \geq 0} (\phi|_{\ell,1}V_m)(\tau, z) e(m\tau')
$$
is a Siegel modular form of weight $\ell$ in Maass space (\cite[pp.74-76]{EZ}).

Notice from (\ref{Vl}) that if $\phi(\tau, 0)=0$ for all $\tau$, then $\phi|_{\ell,1}V_m$ has the same property for all $m$, and thus $v(\phi)|_{z=0}$ is zero. This means that for $F$ in the Maass space, it restricted to $z=0$ is identically zero iff $\phi_1(\tau, 0)$ is identically zero.
Therefore, the subspace of Saito-Kurokawa cusp forms $F$'s that vanish at $z=0$ is isomorphic to the subspace of $J_{\ell,1}^{\mathrm{cusp}}$ that vanish at $z=0$. Now we simply quote the following vector space isomorphism from \cite[p.40]{EZ}:
\begin{align}
 M_{\ell-10} \oplus M_{\ell-12} \cong J_{\ell,1}^{\mathrm{cusp}},
\end{align}
given by the map
\begin{align}
 (f,g) \to f(\tau) \phi_{10,1}(\tau, z) + g(\tau) \phi_{12,1}(\tau, z).
\end{align}
Here $\phi_{10,1}(\tau, z)$ and $\phi_{12,1}(\tau, z)$ are specific Jacobi forms of weight $10$ and $12$ respectively. Moreover
$\phi_{10,1}(\tau, z)= O(z^2)$ while $\phi_{12,1}(\tau, 0)= 12 \Delta (\tau).$
This shows that the dimension of the subspace of $J_{\ell,1}^{\mathrm{cusp}}$ of vanishing forms is the same as $\mathrm{dim}(M_{\ell-10})$. Hence there cannot be more than $\mathrm{dim}(M_{\ell-10})$ Hecke forms that vanishing along $z=0$.
\end{proof}

\begin{proof}[Proof of Theorem \ref{NV2}]
For this result, we use the isomorphism (Theorem \ref{Ja}) between Jacobi forms and half-integral weight forms.  We first consider the special case $D=-4$.  Note that if $c(4)+ 2c(3) \ne 0$ then the coefficient of $q^1$ in the Fourier expansion of $\phi(\tau, 0)$ (which is an elliptic modular form of weight $k$) is nonzero, and hence that $\phi(\tau, 0)$ does not identically vanish.  If $c(4)=-2c(3)$ then $|c(4)|^2=4|c(3)|^2$ and by the Kohnen-Zagier formula \cite[Theorem 1]{KZ}, we deduce that
$$ 4^{\ell-\frac{3}{2}} L(\tfrac{1}{2}, f \otimes \chi_{-4})= 4 \cdot 3^{\ell-\frac{3}{2}} L(\tfrac{1}{2}, f \otimes \chi_{-3}).
$$
We apply the hypothesized lower bound for $L(\frac{1}{2}, f \otimes \chi_{-4})$ and the convexity bound on $L(\frac{1}{2}, f \otimes \chi_{ -3})$ to deduce a contradiction for such an $f$.
The point is the $L$-functions are naturally normalized not to grow too much with respect to $\ell$ while the different growth rate of $4^\ell$ and $3^\ell$ is extreme.

One may observe that the same argument carries through for any $D=-4n$ for which we can show 
\begin{equation}
\label{eq:unequalcn}
 |c(4n)|^2 \neq |\sum_{0 < r^2 \leq 4n} c(4n-r^2)|^2 \leq 4 \sqrt{n} \sum_{0 < r^2 < 4n} |c(4n-r^2)|^2,
\end{equation}
by Cauchy's inequality.  By Kohnen-Zagier, \eqref{eq:unequalcn} would be implied by
$$ (4n)^{\ell-\frac{3}{2}} L(\tfrac{1}{2}, f \otimes \chi_{- 4n}) > 
 4 \sqrt{n} \sum_{0<r^2 < 4n} (4n-r^2)^{\ell-\frac{3}{2}} L(\tfrac{1}{2}, f \otimes \chi_{- (4n-r^2)}),
$$
which in turn would be implied by the stated lower bound on $L(1/2, f \otimes \chi_{-4n})$ and the convexity bound on all other central values, provided $n \ll l^{1-\varepsilon}$.

This argument certainly breaks down once $n$ is approximately $\ell$ for then the terms $4n^{\ell-\frac{3}{2}}$ and $(4n-1)^{\ell-\frac{3}{2}}$ are of roughly the same size. This naturally leads to the question of given $f$, how small can one choose $D$ (in terms of $f$) to guarantee $L(\frac{1}{2}, f \otimes \chi_{- D}) \gg \ell^{-100}?$ This problem was studied by Hoffstein and Kontorovich \cite{HK} who obtained $D$ as small as $\ell^{1+ \varepsilon}$, which barely fails to give a useful lower bound for our application.  The Riemann-Roch theorem shows that there exists a $D \ll \ell$ such that $c(|D|) \neq 0$ but does not give a lower bound on $|c(|D|)|$.
\end{proof}

\subsection{Sketch of proof of Theorem \ref{thm:averageweight}}
\label{section:sketch}
The difficult part of this paper is proving Theorem \ref{thm:averageweight} so here we indicate roughly how the argument goes before embarking on the full proof.
Consider the sum $\sum_{k} w(k) \sum_{f \in B_{2k}} \sum_{g \in B_{k+1}} \omega_f^{-1} \omega_g^{-1} L(1/2, \mathrm{sym}^2 g \otimes f)$, where $\omega_f^{-1}$ and $\omega_g^{-1}$ are Petersson weights; we wish to show that this expression is $\asymp K$.  This corresponds to a variant of the main theorem, namely Theorem \ref{thm:N*} which uses the Petersson weight for $g$ rather than the ``natural'' weight given by Theorem \ref{NF}.

We replace $L(1/2, \mathrm{sym}^2 g \otimes f)$ by $\sum_{n \leq K^{2}} \frac{\lambda_f(n) \lambda_g(n^2)}{\sqrt{n}}$ and apply the Petersson formula in $f$ and in $g$.  The diagonal term leads to the main term, the ``cross terms'' which arise from the diagonal in one of either $f$ or $g$ (but not both) is quite small, and the hard part is to analyze the double sum of Kloosterman sums, namely
\begin{equation*}
 \sum_{n \leq K^{2}} \frac{1}{\sqrt{n}} \sum_{c_1} \sum_{c_2} \frac{S(n^2, 1;c_1) S(n,1;c_2)}{c_1 c_2} \sum_{k \text{ odd}} i^{k} w(k) J_{k}\Big(\frac{4 \pi n}{c_1}\Big) J_{2k-1}\Big(\frac{4 \pi \sqrt{n}}{c_2}\Big).
\end{equation*}
The exponential decay of $J_{\nu}(x)$ for $\nu>0$ large and $x\leq \nu/100$ indicates that $c_1 \ll \frac{n}{K} \ll K^{}$ and $c_2 \ll 1$.  The hard part is to understand this sum of Bessel functions.  In Theorem \ref{thm:Salphabeta} we work out the asymptotics of this sum over $k$; the answer is that roughly
\begin{equation*}
  \sum_{k \text{ odd}} i^{k} w(k) J_{k}\Big(\frac{4 \pi n}{c_1}\Big) J_{2k-1}\Big(\frac{4 \pi \sqrt{n}}{c_2}\Big) \approx e\Big(\frac{2n}{c_1} + \frac{c_1}{4 c_2^2}\Big) (c_1/n)^{1/2}.
\end{equation*}
Since $c_2$ is very small we consider $c_2 = 1$ for simplicity.  Then the double sum of Kloosterman sums is reduced to understanding
\begin{equation*}
 \sum_{n \leq K^{2}} \frac{1}{n} \sum_{c_1 \ll \frac{n}{K}}  \frac{S(n^2, 1;c_1) }{c_1^{1/2}} e\Big(\frac{2n}{c_1}\Big).
\end{equation*}
At this point the Weil bound suffices to show that error term so far is $O(K)$ (in reality it should be $O(K^{1+\varepsilon})$) which is slightly larger than the main term.  Thus the terms with $n \leq K^{1-\varepsilon}$ lead to an error term of size $O(K^{2-\varepsilon})$ which is satisfactory; hence we consider $K^2/2 < n \leq K^2$ and replace $n^{-1}$ by $K^{-1}$.  The summand is then periodic in $c_1$ so is closely approximated by
\begin{equation*}
 K^{-2} \sum_{c_1 \ll K} c_1^{-1/2} \frac{K^2}{c_1} \sum_{r \shortmod{c_1}} S(r^2, 1;c_1) e\Big(\frac{2r}{c_1} \Big).
\end{equation*}
Then we compute
\begin{equation}
 \sum_{r \shortmod{c_1}} S(r^2, 1;c_1) e\Big(\frac{2r}{c_1} \Big) = \sumstar_{h \shortmod{c_1}} \sum_{r \shortmod{c_1}} e\Big(\frac{h(r+\overline{h})^2}{c_1} \Big) = \sumstar_{h \shortmod{c_1}} \sum_{r \shortmod{c_1}} e\Big(\frac{h r^2}{c_1} \Big).
\end{equation}
In \eqref{eq:Tcomp} and following we compute this sum; the answer is that it is $\phi(c_1) c_1^{1/2}$ in case $c_1$ is a square, and vanishes otherwise.  Thus this error term is of size
\begin{equation*}
 \sum_{c_1 \ll K, c_1 = \square} 1 \ll K^{1/2}.
\end{equation*}
This heuristic derivation is consistent with Theorem \ref{thm:N*}.

\section{Acknowledgements}
We warmly thank Akshay Venkatesh for his interest and valuable suggestions.  We also thank Abhishek Saha for some thoughtful suggestions and corrections.

\section{Useful formulas}
\subsection{Approximate functional equation}
The following approximate functional equation was proved for general $L$-functions and can be found in [IK].
\begin{mylemma}[Approximate functional equation] \label{AFE}
Let $H(u)$ be an even holomorphic function with rapid decay as $|\text{Im}(u)| \to \infty$ in a fixed vertical strip, and satisfying
$H(0)=1$, $H(-\frac{1}{4})=0.$ Then we have
\begin{eqnarray} 
 L\left(\tfrac{1}{2}, \mathrm{sym}^2 g \otimes f \right)= 
  2 \sum_{m_1 =1}^{\infty}\sum_{m_2 =1}^{\infty} 
  \frac{\lambda_f (m_1) A_{G}(m_1, m_2)}{(m_1 m_2^2)^{1/2}} V(m_1m_2^2, k),
\end{eqnarray}
where
\begin{eqnarray} \label{V}
 V(y, k)= \frac{1}{2 \pi i} \int_{(3)} y^{-u} H(u) \frac{ \gamma (\frac{1}{2}+ u, k )}{ \gamma (\frac{1}{2}, k)}
   \frac{du}{u},
\end{eqnarray}
and
\begin{eqnarray}
 \gamma(s,k)= 2^3 (2\pi)^{-(3s+3k-\frac{1}{2})} 
\Gamma \left(s+ 2k-\tfrac{1}{2}\right) \Gamma \left(s+k-\tfrac{1}{2} \right) 
  \Gamma \left(s +\tfrac{1}{2} \right). 
\end{eqnarray}
\end{mylemma}
By Stirling's approximation, one shows
\begin{equation} \label{Stirling}
H(u) \frac{ \gamma (\frac{1}{2}+ u, k )}{ \gamma (\frac{1}{2}, k)}= k^{2u}H_1(u)(1+ \frac{c_1(u)}{k} + \dots)
\end{equation}
where $H_1(u)$ is meromorphic on $\mc$ with poles at $u=-1, -2, \dots$, having exponential decay as $|\text{Im}(u)| \to \infty$ in any vertical strip, and each $c_i(u)$ is a polynomial in $u$.  Moving the contour far to the right shows $V(y,k) \ll_A (1 + \frac{y}{k^2})^{-A}$ with $A>0$ arbitrarily large.  Taking the first two terms in the expansion, we write
\begin{equation}
\label{eq:Vexpansion}
V(y,k) = V_0(y/k^2) + k^{-1} V_1(y/k^2) + O(k^{-2} (1 + \frac{y}{k^2})^{-A}),
\end{equation}
where for $i=0,1$, we have
\begin{equation}
\label{eq:Videf}
V_i(x) = \frac{1}{2 \pi i} \int_{(1)} x^{-u} H_1(u) c_i(u) \frac{du}{u}.
\end{equation}

\begin{mylemma} \label{AFRD}
For $i=0,1$, we have
$$
 x^a \frac{\partial^{a}}{\partial x^a} V_i(x) \ll \left(1 + x \right)^{-A}.
$$ The implied constant depends only on $a$ and $A$. 
\end{mylemma}

\subsection{Petersson trace formula.}
The following Petersson trace formula is well-known and can be found in Iwaniec's book \cite{Iw}.
\begin{myprop} [Petersson trace formula] \label{Peter}
Suppose $\phi \in B_{k}$ has the Fourier expansion
$$ \phi(z)= \sum_{n=1}^{\infty} \lambda_{\phi}(n) n^{\frac{k-1}{2}} e(nz)
$$
where $e(z)= e^{2 \pi i z}$ and $\lambda_{\phi}(1)=1$.  Then
$$
\sum_{\phi \in B_{k}} \omega_{\phi}^{-1} \lambda_{\phi}(m) \lambda_{\phi}(n) = \delta_{m,n} + 2 \pi i^{-k}
\sum_{c=1}^{\infty} \frac {S(m,n;c)}{c} J_{k-1} \left( \frac{4\pi \sqrt{mn}}{c} \right),
$$
where 
\begin{equation}
\label{eq:omegadef}
\omega_{\phi}= \frac{(4 \pi)^{k-1} }{\Gamma(k-1)}\| \phi \|^2 = \frac{k-1}{2 \pi^2} L(1, \mathrm{sym}^2 \phi),
\end{equation}
$\delta_{m,n}$ equals $1$ if $m=n$ and $0$ otherwise,
$S(m,n;c)$ is the Kloosterman sum defined below, and $J_{k-1}$ is the $J$-Bessel function.
\end{myprop}
The Kloosterman sum is defined as
$$ S(m,n;c)= \sum_{a\bar{a} \equiv 1 \shortmod{c}} e \left( \frac{ma + n\bar{a}}{c} \right),
$$
and A. Weil proved that $$|S(m,n;c)| \leq (m,n,c)^{\frac{1}{2}} c^{\frac{1}{2}} \tau (c),$$
where $\tau (c)$ is the divisor function.

\section{The norm conjecture}
\label{section:conjecture}
In this section we explain how the moment conjectures of \cite{CFKRS} lead to Conjecture \ref{conj:N2}.
We write the formula for the norm in the form
\begin{equation}
\label{eq:NformulaPetersson}
 N(F_f) = c_f 
\sum_{g \in B_{k+1}} \omega_g^{-1} L(1, \text{sym}^2 g) L(\tfrac12, \text{sym}^2 g \otimes f),
\end{equation}
where $\omega_g = \frac{k}{2 \pi^2} L(1, \text{sym}^2 g)$, and
\begin{equation}
 c_f = \frac{v_2}{v_1^2} \frac{2^2 3 \pi^{-1}}{L(3/2, f) L(1, \text{sym}^2 f)}.
\end{equation}
Now we wish to apply the method of \cite{CFKRS} to find the asymptotic of $N(F_f)$; for the following calculations to make sense, the reader should be familiar with their general recipe.  We use the following representation of the Dirichlet series for $L(s, \text{sym}^2 g \otimes f)$:
\begin{equation}
\label{eq:Lseriesdef}
 L(s, \text{sym}^2 g \otimes f) = \sum_{d, a_1, b_1, a_2, b_2 \geq 1} \frac{\mu(d) \lambda_f(d a_1 b_1^2) \lambda_g(a_1^2) \lambda_g(a_2^2)}{(a_1 b_1^2 a_2^2 d^3 b_2^4)^s},
\end{equation}
which follows from the definition \eqref{eq:RSdef} followed by the relations \eqref{FSym} and \eqref{eq:FSym2} and finally reversal of orders of summation. 

The recipe of \cite{CFKRS} calls for us to {\it formally} write an ``approximate functional equation'' for this $L$-function in the shape
\begin{equation}
L(\tfrac12 + \alpha, \text{sym}^2 g \otimes f) 
= 
\sum_{d, a_1, b_1, a_2, b_2 \geq 1} \frac{\mu(d) \lambda_f(d a_1 b_1^2) \lambda_g(a_1^2) \lambda_g(a_2^2)}{(a_1 b_1^2 a_2^2 d^3 b_2^4)^{1/2 + \alpha}} + \dots ,
\end{equation}
where the dots indicate a similar term with $\alpha$ replaced by $-\alpha$, and multiplied by certain gamma factors (which take the value $1$ when $\alpha=0$).  
Since we shall let $\alpha = 0$ we focus only on this first term.  We also write
\begin{equation}
L(1+2\alpha, \text{sym}^2 g) = \sum_{k, l} \frac{\lambda_g(k^2)}{(kl^2)^{1+2\alpha}}.
\end{equation}
Now we have in this formal way that $c_f^{-1} N(F_f)$ ``is" 
\begin{equation}
\label{eq:momentatalpha}
 \sum_{g \in B_{k+1}}  \omega_g^{-1} \sum_{k,l} \frac{\lambda_g(k^2)}{(k l^2)^{1+2\alpha}} \Big( 
\sum_{d, a_1, b_1, a_2, b_2 \geq 1} \frac{\mu(d) \lambda_f(d a_1 b_1^2) \lambda_g(a_1^2) \lambda_g(a_2^2)}{(a_1 b_1^2 a_2^2 d^3 b_2^4)^{1/2 + \alpha}} + \dots \Big).
\end{equation}
The \cite{CFKRS} conjecture now says we average over $g$ using the Petersson formula and retain only the diagonal.  We need to use the Hecke relation in the form
\begin{equation}
\lambda_g(a_1^2) \lambda_g(a_2^2) = \sum_{c | (a_1^2, a_2^2)} \lambda_g(\frac{a_1^2 a_2^2}{c^2}).
\end{equation}
Thus averaging \eqref{eq:momentatalpha} formally over $g$ with Petersson weights and retaining only the diagonal leaves us with
\begin{equation}
\label{eq:momentmainterm}
\sum_{k,l} \frac{1}{(k l^2)^{1+2\alpha}}  \sum_{d, a_1, b_1, a_2, b_2 \geq 1} \sum_{\substack{c| (a_1^2, a_2^2) \\ c^2 k^2 = a_1^2 a_2^2}} \frac{\mu(d) \lambda_f(d a_1 b_1^2)}{(a_1 b_1^2 a_2^2 d^3 b_2^4)^{1/2 + \alpha}}  + \dots.
\end{equation}
with the dots indicating a term that will be identical to the above when $\alpha = 0$.  Notice that if $c|(a_1^2, a_2^2)$ then $c | a_1 a_2$ so for such $c$ we can always solve for $k = a_1 a_2/c$.  Thus we simplify \eqref{eq:momentmainterm} as $A_f(\alpha) + \dots$, where
\begin{equation}
A_f(\alpha) := \sum_{l,d, a_1, b_1, a_2, b_2 \geq 1} \sum_{c| (a_1^2, a_2^2) } \frac{\mu(d) \lambda_f(d a_1 b_1^2)}{(a_1^3 b_1^2 a_2^4 d^3 b_2^4 l^4/c^2)^{1/2 + \alpha}}.
\end{equation}
Summing freely over $l$ and $b_2$, we see that $A_f(\alpha) = \zeta(2 + 4 \alpha)^2 B_f(\alpha)$, where
\begin{equation}
B_f(\alpha) = \sum_{d, a_1, b_1, a_2 \geq 1} \sum_{c| (a_1^2, a_2^2) } \frac{\mu(d) \lambda_f(d a_1 b_1^2)}{(a_1^3 b_1^2 a_2^4 d^3/c^2)^{1/2 + \alpha}}
\end{equation}
Note that $B_f(\alpha)$ has an Euler product $B_f(\alpha) = \prod_p B_{f,p}(\alpha)$, where with $x = p^{-\half - \alpha}$,
\begin{equation}
 B_{f,p}(\alpha) = \sum_{a_1, b_1, a_2,\geq 0} \sum_{d=0}^{1} \sum_{c \leq \min(2a_1, 2a_2) } (-1)^d \lambda_f(p^{d+ a_1+ 2b_1}) x^{3a_1 + 2 b_1 + 4a_2 +3d   -2c}
\end{equation}
Using $\lambda_f(p^j) = \sum_{r=0}^{j} \alpha_p^r \beta_p^{j-r}$ one can use a computer to explicitly evaluate $B_{f,p}(\alpha)$ as a rational function in $x$.  Indeed, we have
\begin{equation}
 B_{f,p}(\alpha) = \frac{1-x^8}{(1-x^2)(1-\alpha_p^2 x^2)(1-\beta_p^2 x^2) (1-\alpha_p x^3)(1-\beta_p x^3)}.
\end{equation}
Notice that $[(1-\alpha_p^2 x^2)(1-x^2)(1-\beta_p^2 x^2)]^{-1}$ is the Euler factor for the symmetric-square $L$-function associated to $f$.  We then see that
\begin{equation}
 A_f(\alpha) = L(1 + 2 \alpha, \mathrm{sym}^2 f) L(3/2 + 3 \alpha, f) \frac{\zeta(2 + 4 \alpha)^2}{\zeta(4 + 8 \alpha)}.
\end{equation}
We obtain the conjectured on the asymptotic for $N(F_f)$ by setting $\alpha = 0$ and multiplying by $2$ to account for the term with $-\alpha$.  Thus we are led to
\begin{equation}
N(F_f) = 2 c_f L(1, \mathrm{sym}^2 f) L(3/2, f)
\frac{\zeta(2)^2}{\zeta(4)} + O(k^{-\delta}),
\end{equation}
for some $\delta > 0$.  Using \eqref{eq:volumes}, we have $2 c_f L(1, \mathrm{sym}^2 f) L(3/2, f) \zeta(2)^2 \zeta^{-1}(4) = 2$, giving Conjecture \ref{conj:N2}.

\section{Summation formulas}

\subsection{Summing over the weight}
In this section we will state some summation formulas which we need in the proof of Theorem \ref{thm:averageweight}.
\begin{myprop} [\cite{Iw}, pp.87-88] \label{A}
Let $K \geq 1$, and suppose $w$ satisfies \eqref{eq:wprop}.  For $a=0, 2$ and $x >0$, we have
\begin{equation}
\label{eq:ksumoneJ}
4 \sum_{k \equiv a \shortmod{4}} w\left( k-1 \right) J_{k-1}(x)= w(x)-i^{a} g(x) + O \left(\frac{x}{K^3} \right)
\end{equation}
where
\begin{equation*}
 g(x)  =\frac{1}{\sqrt{x}} \text{Im} \Big( e^{ix- \pi i/4} 
     \check{w} \big( \frac{1}{2x} \big) \Big), 
\end{equation*}
with
$$ \check{w}(v) =\int_0^{\infty} \frac{w(\sqrt{u})}{\sqrt{2\pi u}} e^{iuv}du. 
$$
Moreover $g^{}(x) \ll_j (xK^{-2})^j$ for any $j \geq 0.$
\end{myprop}
Prior to summing over $k$, the left hand side of \eqref{eq:ksumoneJ} is very small for $x \leq K/100$, say, using the exponentially small size of the Bessel functions in this regime (see \ref{eq:Besseltrivial} below).  After summing over $k$, we have two terms.  One term retains the condition $x \approx K$ but is not oscillatory while the other term (involving $g$) is oscillatory but effectively has $x \gg K^{2+\varepsilon}$ which is a much more restrictive truncation.

\begin{mycoro}  \label{A1} We have
 $$ 2\sum_{k \equiv 0 \shortmod{2}} i^k w \left( k-1  \right) J_{k-1}(x)
  = - \frac{2}{\sqrt{x}} \text{Im} \Big( e^{ix- \pi i/4} 
     \check{w} \big( \frac{1}{2x} \big) \Big) + O \left( \frac{x}{K^3} \right).
$$
\end{mycoro}

Suppose that $\alpha, \beta$ are positive real numbers, $w$ satisfies \eqref{eq:wprop}, and let
\begin{equation}
\label{eq:Sdef}
S(\alpha, \beta):= \sum_{k \text{ odd}} i^k J_{k}(4 \pi \alpha) J_{2k-1}(4 \pi \beta) w(k).
\end{equation}
Our goal is to understand the asymptotic behavior of $S(\alpha,\beta)$ for $\alpha, \beta$ in different ranges.  It turns out that the answer strongly depends on the size of $\frac{\beta}{4 \alpha}$.

\begin{mytheo}
\label{thm:Salphabeta}
 If $\alpha$ or $\beta$ is smaller than $K/100$ then $S(\alpha, \beta) \ll \exp(-\half K)$.  Otherwise, suppose $\alpha \ll K^{2+\varepsilon}$ and $\beta \ll K^{1+\varepsilon}$, and let $\gamma = \frac{\beta}{4 \alpha}$.  If $\gamma \geq 1$ then $S(\alpha, \beta) \ll K^{-1}$.  
For $\gamma < 1$ we have the asymptotic formula
\begin{equation}
\label{eq:Sasymp}
S(\alpha, \beta) = \sum_{\pm } \frac{H_{\pm}(\beta \sqrt{1-\gamma^2}, \gamma)}{\sqrt{\alpha}} e(\pm (2 \alpha + \frac{\beta^2}{4 \alpha})) + O(K^{-1+\varepsilon}),
\end{equation}
where $H_{\pm}(x,y)$ is a complex-valued function vanishing unless $K \leq 2 \pi x \leq 2K$ and $0 \leq y \leq 1$, and satisfying
\begin{equation}
H_{\pm}^{(i,j)}(x,y) \ll \frac{\beta}{K} K^{-i} , 
\end{equation}
for $i, j \in \{0, 1, 2, \dots \}$, and $A > 0$ is arbitrary, where the implied constants depend only on $i, j, A$.
\end{mytheo}
The answer is in contrast to the more well-known such sum involving one Bessel function (Proposition \ref{A}).
Notice here that the answer is oscillatory and has its largest size for $\alpha$ of rough size $K$.

We develop a fairly explicit formula for $H_{\pm}$ in the course of the proof (specifically \eqref{eq:H-formula}) which clearly displays its dependency on any auxiliary parameters.

The proof is somewhat lengthy and we divide it into managable pieces. We will prove Theorem \ref{thm:Salphabeta} in Section \ref{section:Ksum}.

\subsection{Poisson summation.}
Suppose that $W(n,k)$ is a smooth function with support on $[N, 2N] \times [K, 2K]$ and satisfing for any $j, j' \in \{ 0,1, 2 \dots \}$, and any $A > 0$
\begin{equation}
\frac{\partial^j}{\partial x^j} \frac{\partial^{j'}}{\partial x^{j'}} W(x,y) \ll_{j,j',A} (1 + \frac{x}{N})^{-A} ( 1 + \frac{y}{K})^{-A}. 
\end{equation}
Define
\begin{equation} \label{SumR}
T= \sum_{n=1}^{\infty} S(n^2, r_2^2;c_1) S(r_1 n, 1 ;c_2) \sum_{k \text{ odd}} i^k W(n,k) J_k(4 \pi \frac{n r_2}{c_1}) J_{2k-1}(4 \pi \frac{\sqrt{n} \sqrt{r_1}}{c_2}).
\end{equation}
Here $r_1, r_2, c_1, c_2$ are certain integers, and $N r_1 r_2^2 \ll K^{2+\varepsilon}$.  Let $\alpha = \frac{n r_2}{c_1}$, $\beta = \frac{\sqrt{n r_1}}{c_2}$.  In our application $\alpha \ll K^{2 + \varepsilon}$, and $\beta \ll K^{1 + \varepsilon}$.  Thus we can apply Theorem \ref{thm:Salphabeta} to evaluate the sum over $k$.  Write $T = \sum_{\pm} T_{\pm} + T_E$ where $T_\pm$ correspond to the two main terms, and $T_E$ corresponds to the error term.  
We will prove the following Theorem \ref{thm:Poisson} in Section \ref{section:PoissonProof}.
\begin{mytheo} \label{thm:Poisson}
Let  notations and conditions be as above. Suppose that $c_2 \ll K^{\varepsilon}, c_1 \ll K^{1+\varepsilon}$. Write $c_1= c_0 c'$ where $(c_0,c')=1$ and $c'| c_2^{\infty}.$ Then we have
 \begin{equation} \label{TE}
T_E \ll K^{-1+\varepsilon} c_2 c_1^{1/2+\varepsilon} N
\end{equation}
and
\begin{equation} 
 T_{\pm} \ll \sqrt{N} c_2^{1/2+ \varepsilon} (c_0 c')^{1+\varepsilon}  \delta_{c_0=\square} + O(K^{-100}),
\end{equation}
where $\delta_{m=\square}$ is the indicator function for $m=\square$.
\end{mytheo}

\section{Asymptotic evaluation of the restriction norm on average}
\label{section:averagenorm}
Let notation be as in Section \ref{section:L2}.
\begin{mytheo} \label{thm:N*}
Let $N^*(F_f)$ be given by 
\begin{equation}
\label{eq:N*}
 N^*(F_f) = c_f'
\sum_{g \in B_{k+1}} \omega_g^{-1}  L(\tfrac12, \mathrm{sym}^2 g \otimes f),
\end{equation}
where $\omega_g = \frac{k}{2 \pi^2} L(1, \mathrm{sym}^2 g)$, and
\begin{equation}
 c_f' = \frac{v_2}{v_1^2} \frac{2^2 3 \pi^{-1}}{L(1, \mathrm{sym}^2 f)}.
\end{equation}
Then for $w(k)$ satisfying \eqref{eq:wprop}, we have
\begin{equation}
\sum_{k \text{ odd}} w(k) \sum_{f \in B_{2k}} N^*(F_f) = \frac{4}{5} \sum_{k \text{ odd}} w(k) \dim(S_{2k})  + O(K^{3/2+\varepsilon}).
\end{equation}
\end{mytheo}
Observe that $N^*(F_f)$ is given by the same definition as $N(F_f)$ except we have removed the weight $L^{-1}(3/2, f) L(1, \mathrm{sym}^2 g)$.  Thus $N^*(F_f) (\log{K})^{-1} \ll N(F_f) \ll N^*(F_f) \log{K}$.  In Section \ref{section:extraweights} below we explain how to re-introduce these weights and recover Theorem \ref{thm:averageweight}.

\begin{proof}
We have $\omega_f = \frac{2k-1}{2 \pi^2} L(1, \mathrm{sym}^2 f)$, and hence
\begin{equation}
\sum_{f \in B_{2k}} N^*(F_f) = \frac{\vol(\mathcal{F}_2)}{\vol(\mathcal{F}_1)^2} 6 \pi^{-3} (2k-1) \sum_{f \in B_{2k}} \sum_{g \in B_{k+1}} \omega_f^{-1} \omega_g^{-1} L(\tfrac12, \mathrm{sym}^2 g \otimes f).
\end{equation}
 Letting $u(k) = \frac{2k-1}{K} w(k)$, and $c'' = \frac{6}{\pi^3} \frac{v_2}{v_1^2}$, we then have
 \begin{equation}
 \label{eq:N2*}
\mathcal{N}^* :=\sum_{k \text{ odd}} w(k) \sum_{f \in B_{2k}} N^*(F_f) = K c'' \mathcal{M}^*,
\end{equation}
where
\begin{equation}
\mathcal{M}^* := \sum_{k \text{ odd}} u(k) \sum_{f \in B_{2k}} \sum_{g \in B_{k+1}} \omega_f^{-1} \omega_g^{-1} L(\tfrac12, \mathrm{sym}^2 g \otimes f).
 \end{equation}
 Note that $u(k)$ satisfies \eqref{eq:wprop}.
 
 By the approximate functional equation (Lemma \ref{AFE}), and \eqref{eq:Lseriesdef}, we have
\begin{eqnarray*}
 L\left(\tfrac{1}{2}, \mathrm{sym}^2 g \otimes f \right) 
 = 2 \sum_{d \geq 1} \frac{\mu(d)}{d^{3/2}} \sum_{a_1 , b_1, a_2, b_2 \geq 1}
  \frac{\lambda_f (da_1 b_1^2)}{(a_1b_1^2 a_2^2 b_2^4)^{1/2}}V(d^3 a_1b_1^2 a_2^2 b_2^4, k) 
   \lambda_g (a_1^2) \lambda_g(a_2^2).
\end{eqnarray*}
Using \eqref{eq:Vexpansion} we write $\mathcal{M}^* = \mathcal{M}_0 + \mathcal{M}_1 +\mathcal{M}_E$, corresponding to the three terms on the right hand side of \eqref{eq:Vexpansion}.  Trivially bounding the sum of Hecke eigenvalues (say using Deligne's bound) immediately shows $\mathcal{M}_E \ll K^{\varepsilon}$.  The form of $\mathcal{M}_1$ is practically the same as $\mathcal{M}_0$ but is $1/K$ times the size, and so the method used to show $\mathcal{M}_0 \ll K$ then would show $\mathcal{M}_1 \ll 1$.  We henceforth consider only $\mathcal{M}_0$.

By applying the Petersson trace formula twice, we get
\begin{eqnarray*}
\mathcal{M}_0 = \sum_{k \text{ odd}} u(k) (M_k + E_{1,k} + E_{2,k} + E_{3,k}),
\end{eqnarray*}
where
\begin{equation}
M_k = 2\sum_{b_2 \geq 1} \frac{1}{b_2^2} V_0(b_2^4/k^2),
\end{equation}
\begin{multline}
E_{1,k} =
8  \pi^2 i^{-3k-1} \sum_{d \geq 1} \frac{\mu(d)}{d^{3/2}}
  \sum_{a_1 , b_1, a_2, b_2 \geq 1} \frac{V_0(d^3 a_1b_1^2 a_2^2 b_2^4/k^2) }{(a_1b_1^2 a_2^2 b_2^4)^{1/2}} \times 
\\
   \sum_{c_1 \geq 1} \frac{S(a_1^2, a_2^2; c_1)}{c_1}J_{k} \left( \frac{4\pi \sqrt{a_1^2a_2^2}}{c_1} \right) 
  \sum_{c_2 \geq 1} \frac{S(da_1b_1^2, 1; c_2)}{c_2}J_{2k-1} \left( \frac{4\pi \sqrt{d a_1b_1^2}}{c_2} \right),
\end{multline}
\begin{equation}
E_{2,k} = 4 \pi i^{-(k+1)} \sum_{a_2, b_2 \geq 1} \frac{V_0( a_2^2 b_2^4/k^2)}{(a_2^2 b_2^4)^{1/2}} 
   \sum_{c_1 \geq 1} \frac{S(1, a_2^2; c_1)}{c_1} 
   J_{k} \left( \frac{4\pi \sqrt{a_2^2}}{c_1} \right),
\end{equation}
and
\begin{multline}
E_{3,k} = 4\pi i^{-2k} \sum_{d \geq 1} \frac{\mu(d)}{d^{3/2}}  \sum_{a_1 , b_1, b_2 \geq 1} 
      \frac{V_0(d^3 a_1^3 b_1^2 b_2^4/k^2)}{(a_1^3 b_1^2 b_2^4)^{1/2}}  
   \sum_{c_2 \geq 1}  \frac{ S(da_1 b_1^2, 1; c_2) }{c_2} J_{2k-1} \left( \frac{4\pi \sqrt{d a_1b_1^2} }{c_2} \right).
\end{multline}
Let $\displaystyle M= \sum_{k \text{ odd}} u \left(k \right) M_{k}$ and let $\displaystyle E_i= \sum_{k \text{ odd}} u \left(k \right) E_{i,k}$ for $i=1,2, 3$.
In the following Lemmas, we will show
$M_k = 2\zeta(2)  + O(K^{-2+\varepsilon}), \
 E_1  \ll K^{1/2 + \varepsilon}, \ E_2 \ll K^{\varepsilon}$ and $E_3 \ll K^{ \varepsilon}$, which 
then will show
\begin{equation}
\mathcal{N}^* =  2 \zeta(2) c'' \sum_{k \text{ odd}}  (2k-1) w(k) + O(K^{3/2 + \varepsilon}).
\end{equation} 
Recalling $\dim(S_{2k}) = (2k-1)/12 + O(1)$, we then express this main term as
\begin{equation}
24  c'' \zeta(2)  \sum_{k \text{ odd}} \dim(S_{2k}) w(k).
\end{equation}
Finally observe
\begin{equation}
24 c'' \zeta(2) = 2^3 3 \frac{6}{\pi^3} \frac{\pi^3/270}{(\pi/3)^2} \frac{\pi^2}{6} = \frac{4}{5}. \qedhere
\end{equation}
\end{proof}

\begin{mylemma} For  $k$ large enough, we have
 $$ M_k= 2 \zeta (2) + O(k^{-2+\varepsilon}),
$$
\end{mylemma}
\begin{proof}
Using \eqref{eq:Videf} and then changing the order of summation and integration, we have
\begin{eqnarray*}
 M_k =2\cdot \frac{1}{2\pi i} \int_{(3)} \zeta (2 + 4u) k^{2u} H_1(u)
      \frac{du}{u}.
\end{eqnarray*}
Recall that $H(0) = 1$ and $H(-1/4) = 0$ and hence $H_1(u)$ has the same properties (see \eqref{Stirling}).  Therefore, moving the line of integration to $(-1+\varepsilon)$ we pick up a simple pole at $u=0$ only
which delivers the stated main term.  The error term is obtained by trivially bounding the new contour integral.
\end{proof}

\begin{mylemma} We have
 $$ E_2 \ll K^{\varepsilon}, \quad and \quad E_3 \ll K^{ \varepsilon}.
$$
\end{mylemma}
\begin{proof}  Consider $E_3$. We have 
\begin{eqnarray*}
 E_3 = 4\pi \sum_{ \substack{ d, a_1 , b_1, b_2 \geq 1 \\  d^3 a_1^3 b_1^2 b_2^4 \ll K^{2+ \varepsilon}}} 
 \frac{\mu(d)}{d^{3/2}} \frac{1}{(a_1^3 b_1^2 b_2^4)^{1/2}} \sum_{c_2 \geq 1} \frac{ S(da_1 b_1^2, 1; c_2) }{c_2} F_3+ O(K^{-100}),  
\end{eqnarray*}
where
\begin{eqnarray*}
 F_3=  \sum_{k \text{ odd}} i^{-2k} u(k) V_0(d^3 a_1^3 b_1^2 b_2^4/k^2)J_{2k-1} \left( \frac{4\pi \sqrt{d a_1b_1^2} }{c_2} \right).
\end{eqnarray*}
The change of variables $2k-1 \to k-1$ gives
\begin{eqnarray*}
F_3 = - \sum_{k \equiv 2 \shortmod{4}} u \left( \frac{k}{2} \right) V_0(4 d^3 a_1^3 b_1^2 b_2^4/k^2)J_{k-1} \left( \frac{4\pi \sqrt{d a_1b_1^2} }{c_2} \right).
\end{eqnarray*}
We claim that $F_3 \ll 1$ and is very small $(O(K^{-100}))$ unless
\begin{eqnarray} \label{cond}
 c_2da_1b_2 \ll K^{\varepsilon}.
\end{eqnarray}
First observe that $V_0(4d^3 a_1^3 b_1^2 b_2^4/k^2)$ is very small unless $d^3 a_1^3 b_1^2 b_2^4 \ll K^{2+\varepsilon}$, which we henceforth assume.
By \cite[2.11 (5)]{Wa} 
\begin{equation}
\label{eq:Besseltrivial}
J_{k-1}(x) \ll 2^{-k}x, \quad  \text{if} \quad k \geq 2 \text{ and } 0< x \leq k/3.
\end{equation}
In our case, $$x= \frac{4\pi \sqrt{d^3a_1^3b_1^2b_2^4}}{c_2da_1b_2^2} \ll \frac{K^{1+ \varepsilon}}{c_2da_1b_2^2}.$$
This immediately means $F_3$ is very small unless $(\ref{cond})$ holds. 
Thus we have by Proposition \ref{A} that $F_3 \ll 1$ and so using \eqref{cond} we immediately obtain $E_3 \ll K^{\varepsilon}$.

The proof for $E_2$ is similar, so we omit the details. The point is that essentially $a_2 \ll K^{1+\varepsilon}$ so \eqref{eq:Besseltrivial} essentially gives $c_1 \ll K^{\varepsilon}$, while Corollary \ref{A1} shows the sum over $k$ is bounded.
\end{proof}

%
%

\begin{mylemma} We have
$ E_1 \ll K^{\frac{1}{2}+ \varepsilon}.
$
\end{mylemma}
\begin{proof}
 By Lemma \ref{AFRD},
\begin{multline*}
 E_1= 8 \pi^2 i^{-1} \sum_{\substack{d, a_1, b_1, a_2, b_2 \geq 1 \\ d^3a_1b_1^2a_2^2b_2^4 \ll K^{2+ \varepsilon}}}
      \frac{\mu(d)}{d^{3/2}} \frac{1}{(a_1 b_1^2a_2^2b_2^4)^{1/2}} \sum_{c_1 \geq 1} \frac{S(a_1^2, a_2^2; c_1)}{c_1}
\sum_{c_2 \geq 1} \frac{S(da_1b_1^2, 1; c_2)}{c_2} \\
\sum_{k \text{ odd}} i^{k} u \left(k \right) V_0(d^3a_1b_1^2a_2^2b_2^4/ k^2) J_{k} \left( \frac{4\pi a_1a_2}{c_1} \right) J_{2k-1} \left( \frac{4\pi \sqrt{d a_1b_1^2}}{c_2} \right) + O(K^{-100}).
\end{multline*}
For $c_1 \gg K^{1+2\varepsilon}$ or $c_2 \gg K^{\varepsilon}$, we have 
$$ \alpha = \frac{a_1a_2}{c_1} \ll \frac{K^{2+\varepsilon}}{K^{1+2\varepsilon}}= K^{1-\varepsilon}
$$
or
$$ \beta= \frac{\sqrt{da_1b_1^2}}{c_2} \ll \frac{K^{1+ \varepsilon/2}}{K^{\varepsilon}} \ll K^{1-\frac{\varepsilon}{2}}.
$$
By Theorem \ref{thm:Salphabeta}, the contribution to $E_1$ from $c_1 \gg K^{1+2\varepsilon}$ or $c_2 \gg K^{\varepsilon}$ is thus negligible.  Next we apply a smooth dyadic partition of unity to the sum over $a_1$, with say $1 = \sum_{N \text{ dyadic }} \Omega(a_1/N)$, where $\Omega$ is a fixed smooth function with support contained in $[1,2]$.  It then suffices to show $E_N^* \ll K^{1/2 + \varepsilon}$, for each $1 \ll N \ll K^{2 +\varepsilon}$, where
\begin{equation*}
 E^*_N := N^{-1/2} \sum_{\substack{d, b_1, a_2, b_2 \geq 1 \\ d^3b_1^2a_2^2b_2^4N \ll K^{2+ \varepsilon}}}
      \frac{\mu(d)}{d^{3/2}} \frac{1}{( b_1^2a_2^2b_2^4)^{1/2}} 
\sum_{c_1 \ll K^{1+ 2\varepsilon}} \frac{1}{c_1} 
 \sum_{c_2 \ll K^{\varepsilon}} \frac{1}{c_2} T^*,
\end{equation*}
where $T^*$ is
an instance of the expression $T$ defined by \eqref{SumR}, with $n=a_1$, $r_1 = db_1^2$, $r_2 = a_2$, and $W(n,k) = \Omega(n/N) (n/N)^{-1/2} V_0(n r_1 d^2 r_2^2 b_2^4/k^2) u(k)$.
By Theorem \ref{thm:Poisson}, we have $T=T_{\pm}+ T_E$.
The contribution to $E^*_N$ from $T_E$ is at most
\begin{multline*}
N^{-1/2}\sum_{\substack{d, b_1, a_2, b_2 \geq 1 \\ d^3b_1^2a_2^2b_2^4N \ll K^{2+ \varepsilon}}}
      \frac{1}{d^{3/2}} \frac{1}{( b_1^2a_2^2b_2^4)^{1/2}} 
\sum_{c_1 \ll K^{1+ 2\varepsilon}} \frac{1}{c_1} 
\sum_{c_2 \ll K^{\varepsilon}} \frac{1}{c_2} K^{-1+\varepsilon}c_2c_1^{1/2+\varepsilon}N 
\ll N^{1/2} K^{-1/2 + \varepsilon},
\\
\end{multline*}
which is satisfactory since $N \ll K^{2+\varepsilon}$.
By Theorem \ref{thm:Poisson}, the contribution from $T_{\pm}$ to $E^*_{N}$ is at most $O(K^{-50})$ plus 
\begin{multline*}
\ll N^{-1/2}\sum_{\substack{d^3b_1^2a_2^2b_2^4N \ll K^{2+ \varepsilon}}}
    \frac{1}{( d^3 b_1^2a_2^2b_2^4)^{1/2}} 
\sum_{c_2 \ll K^{\varepsilon}} \frac{1}{c_2} 
\sum_{\substack{ (c_0,c')=1 \\ c_0 c' \ll K^{1+\varepsilon} \\ c'|c_2^{\infty} }} 
\frac{1}{c_0 c'} N^{1/2} c_2^{1/2+ \varepsilon} (c_0 c' )^{1+\varepsilon}\delta_{c_0=\square},
\end{multline*} 
which is then seen to be $O(K^{1/2 + \varepsilon})$ by trivial estimations.
\end{proof}

\section{Proof of Theorem \ref{thm:averageweight}}
\label{section:extraweights}
Recall in Theorem \ref{thm:N*} we showed an asymptotic for the mean value of $N^*(F_f)$ with a power saving, where $N^*(F_f)$ is the same as $N(F_f)$ but with slightly different weights.  Precisely, we need to include $L(1, \text{sym}^2 g)/L(3/2, f)$ to get $N(F_f)$.  

We argue that $L(1, \text{sym}^2 g)$ can almost always be approximated by a short sum.  Precisely, write
\begin{equation}
\label{eq:L1short}
L(1, \text{sym}^2 g) = \sum_{q, r} \frac{\lambda_g(r^2)}{q^2 r} \exp(-q^2 r/V) + \mathcal{E}(g, V),
\end{equation}
Let $\delta, \varepsilon_1 > 0$ be chosen (small), and set $V = K^{\delta}$.  Then there exists $\varepsilon_2 > 0$ such that the number of pairs $(g, k)$ such that $g \in B_{k+1}$, and $K \leq k \leq 2K$ such that $|\mathcal{E}(g, V)| \geq K^{-\varepsilon_2}$ is $O(K^{\varepsilon_1})$.  Furthermore, by trivial estimations we have $|\mathcal{E}(g, V)| \ll (\log{K})^{100}$, say.  This result is a consequence of a zero-density estimate for this family which is a very minor modification of results of \cite{Luo} who considered the harder case of Maass forms on the full modular group.

Now modifying \eqref{eq:N2*}, we have
\begin{equation}
 \label{eq:N}
\mathcal{N} := \sum_{k \text{ odd}} w(k) \sum_{f \in B_{2k}} N(F_f) = K c'' \mathcal{M},
\end{equation}
where
\begin{equation}
\mathcal{M} := \sum_{k \text{ odd}} u(k) \sum_{f \in B_{2k}} \sum_{g \in B_{k+1}} \omega_f^{-1} \omega_g^{-1} \frac{L(1, \text{sym}^2 g)}{L(3/2, f)} L(\tfrac12, \mathrm{sym}^2 g \otimes f).
 \end{equation}
Next we apply the approximation \eqref{eq:L1short}.  The error term to $\mathcal{M}$ in this approximation is
\begin{equation}
\label{eq:L1approx}
\ll \sum_{k \text{ odd}} u(k) \sum_{f \in B_{2k}} \sum_{g \in B_{k+1}} \omega_f^{-1} \omega_g^{-1} |\mathcal{E}(g, V)|  L(\tfrac12, \mathrm{sym}^2 g \otimes f),
\end{equation}
by positivity of the central values and the trivial bound $L(3/2, f)^{-1} \ll 1$.  The contribution to \eqref{eq:L1approx} from ``good" $g, k$, i.e. those such that $|\mathcal{E}(g, V)| \leq K^{-\varepsilon_2}$ is, by Theorem \ref{thm:N*}, 
\begin{equation}
\ll K^{-\varepsilon_2}   
\sum_{k \text{ odd}} u(k) \sum_{f \in B_{2k}} \sum_{g \in B_{k+1}} \omega_f^{-1} \omega_g^{-1}  L(\tfrac12, \mathrm{sym}^2 g \otimes f) = K^{-\varepsilon_2}  \mathcal{M}^*
\ll K^{1-\varepsilon_2}.
\end{equation}
The contribution to \eqref{eq:L1approx} from ``bad" $g, k$ is, using the convexity bound $L(1/2, \text{sym}^2 g \otimes f) \ll K$, 
\begin{equation}
\ll K^{\varepsilon_1} (\log{K})^{100} K^{-2+\varepsilon} K^2  \ll K^{2 \varepsilon_1}.
\end{equation}
keeping in mind the Petersson weights satisfy $\omega_{f}^{-1}, \omega_g^{-1} \ll K^{-1+\varepsilon}$ (see \eqref{eq:omegadef}).  Thus for purposes of obtaining the asymptotic formula for $\mathcal{N}$, it suffices to replace $L(1, \text{sym}^2 g)$ by the sum on the right hand side of \eqref{eq:L1short}.

Next we observe
\begin{equation}
\frac{1}{L(s,f)} = \prod_p(1 - \frac{\lambda_f(p)}{p^s} + \frac{1}{p^{2s}}) = \sum_{a, b} \frac{\mu^2(ab) \mu(a) \lambda_f(a)}{(ab^2)^{s}} =: \sum_{t=1}^{\infty} \frac{\mu_f(t)}{t^s}.
\end{equation}
A trivial bound gives
\begin{equation}
\label{eq:1/L}
\frac{1}{L(3/2, f)} = \sum_{t=1}^{\infty} \frac{\mu_f(t)}{t^{3/2}} \exp(-t/V) + O(V^{-\half + \varepsilon}).
\end{equation}

By combining \eqref{eq:L1short} and \eqref{eq:1/L}, we then find a satisfactory approximation to 
$\mathcal{N}$ which leads to a minor modification of the sums studied in Section \ref{section:averagenorm}.  As before, we apply an approximate functional equation and apply the Petersson trace formula over $f$ and $g$.  The error terms are practically of the same shape as before (since $q, r, t$ are small) and so we claim that the asymptotic for $\mathcal{N}$ comes from the diagonal only, which we denote $\mathcal{N}_0$; we also let $\mathcal{M}_0$ be the corresponding diagonal term from $\mathcal{M}$.  We shall presently show
\begin{equation}
\label{eq:M0asymp}
\mathcal{M}_0 \sim 
2 \frac{\zeta(2)^3}{\zeta(4)}. 
\end{equation}
\begin{proof}[Proof of \eqref{eq:M0asymp}]

We have
\begin{equation}
\mathcal{M}_0 = 2 \sum_{d, a_1, b_1, a_2, b_2, q, r, a, b} \frac{\mu(d)\mu^2(ab) \mu(a) }{(d^3 a_1 b_1^2 a_2^2 b_2^4 q^4 r^2 a^3 b^6)^{1/2}} \delta_{d a_1 b_1^2 = a} \sum_{c | (a_1^2, a_2^2)} \delta_{r^2 = \frac{a_1^2 a_2^2}{c^2}} (\dots),
\end{equation}
where $\delta_{m=n}$ is $1$ if $m=n$ and is $0$ otherwise, and the $(\dots)$ indicates only the weight functions.  In the course of the computations it will be apparent that the sum above converges absolutely.  The Perron formula method will then show that $\mathcal{M}_0$ is asymptotic to the above sum with the weight functions removed (we use the fact that the Mellin transforms of the weight functions all have residue $1$ at the origin).  Now we observe that $b_2$ and $q$ are free and each leads to a $\zeta(2)$, and that $b_1 = 1$ since $a$ is squarefree.  Furthermore, we solve $r = a_1 a_2/c$ and $a = d a_1$ so that
\begin{equation}
\mathcal{M}_0 \sim 2 \zeta(2)^2 \sum_{d, a_1,  a_2,  b} \sum_{c | (a_1^2, a_2^2)} c \frac{\mu(d)\mu^2(d a_1 b) \mu(d a_1) }{(d^6 a_1^6 a_2^4 b^6)^{1/2}}.
\end{equation}
Taking the Euler product of the above, we have
\begin{equation}
\mathcal{M}_0 \sim 2 \zeta(2)^2 \prod_{p} \Big(\sum_{\substack{0 \leq d, a_1,  a_2,  b \\ 0 \leq d + a_1 +b \leq 1}}  \sum_{0 \leq c \leq \min(2 a_1, 2a_2)} p^c \frac{(-1)^{a_1} }{p^{3d + 3a_1 + 2a_2 + 3b}} \Big).
\end{equation}
Note that $0 \leq c \leq 2$.  The contribution to the inner sums above from $c=0$ is
\begin{equation}
\label{eq:c0}
\sum_{\substack{0 \leq d, a_1,  a_2,  b \\ 0 \leq d + a_1 +b \leq 1}} \frac{(-1)^{a_1} }{p^{3d + 3a_1 + 2a_2 + 3b}} = (1 - p^{-2})^{-1} (1 + p^{-3} - p^{-3} + p^{-3})
\end{equation}
If $c =1$ or $c=2$ then $a_1 =1$, $d=b=0$, and the condition on $a_2$ is $a_2 \geq 1$.  Thus these terms contribute
\begin{equation}
\label{eq:c12}
- p^{-3} \sum_{c=1}^{2} p^c \sum_{a_2 \geq 1} \frac{1}{p^{2a_2}} = -p^{-5} (1-p^{-2})^{-1} (p + p^2).
\end{equation}
Note that $\eqref{eq:c0}$ plus $\eqref{eq:c12}$ is $(1-p^{-2})^{-1} (1-p^{-4})$. Thus \eqref{eq:M0asymp} holds, as desired.
\end{proof}

Now applying \eqref{eq:M0asymp} to $\mathcal{N}$, we have 
\begin{equation}
\mathcal{N} \sim 24 c'' \frac{\zeta(2)^3}{\zeta(4)} \sum_{k \text{ odd}} w(k) \dim(S_{2k}).
\end{equation}
Finally we finish the proof of Theorem \ref{thm:averageweight} by computing
\begin{equation}
24 c'' \frac{\zeta(2)^3}{\zeta(4)} = 2^3 3 \frac{2 3}{\pi^3} \frac{\pi^3}{270} \frac{3^2}{\pi^2} \frac{\pi^6}{2^3 3^3} \frac{90}{\pi^4} = 2.
\end{equation}

\section{Proof of Theorem \ref{thm:Salphabeta}}
\label{section:Ksum}
As a first step, we remark that the exponential decay of $S(\alpha, \beta)$ for $\alpha$ or $\beta$ $\leq K/100$ follows from the rapid decay of the Bessel functions, that is \eqref{eq:Besseltrivial}.

Our next step is to find a compact integral representation for $S(\alpha, \beta)$ that immediately shows $S(\alpha, \beta) \ll 1$, which is already a sizable savings compared to say using the  bound $|J_{k}(x)| \ll k^{-1/3}$.
\begin{mylemma}
Recall the definition of $S(\alpha, \beta)$ given by \eqref{eq:Sdef}. 
 We have
\begin{equation}
\label{eq:SIntegralForm}
 S(\alpha, \beta) = \half \sum_{\pm} \pm  \int_{-1/2}^{1/2} e(-2 \beta \sin(2 \pi u)) e(-u) I_{K,\pm \alpha}(u) du,
\end{equation}
where
\begin{equation}
 I_{K,\pm \alpha}(u) = \intR e(\pm 2 \alpha \cos(2 \pi (2u-t))) \widehat{w}(-t) dt.
\end{equation}
\end{mylemma}
The nice feature of \eqref{eq:SIntegralForm} is that it effectively involves only one integral, since we shall find an asymptotic formula for $I_{K, \pm \alpha}(u)$ in Lemma \ref{lemma:IK} below.

\begin{proof}
We begin by applying the integral representation
\begin{equation}
 J_{l}(x) = \int_{-\half}^{\half} e(lt) e^{-ix \sin(2 \pi t)} dt,
\end{equation}
to both Bessel functions and reversing the orders of summation and integration, giving
\begin{equation}
 S(\alpha, \beta) = \int_{-1/2}^{1/2} \int_{-1/2}^{1/2} e(-2\alpha \sin(2 \pi t) - 2 \beta \sin(2 \pi u)) e(-u) T_K(t,u) dt du,
\end{equation}
where
\begin{equation}
 T_K(t,u) = \sum_{k \text{ odd}} i^k e(k(t + 2u)) w(k).
\end{equation}
Next we apply the Poisson summation formula to the sum over $k$.  After some calculation, we find that for any $y \in \mr$,
\begin{equation}
 \sum_{k \text{ odd}} i^k w(k) e(ky) = i e(y) \sum_{\nu \in \mz} \widehat{w_{1}}(\nu - 2y -\half),
\end{equation}
where $w_{1}(x) = w(1+2x)$.
Thus, after interchanging the orders of summation and integration, we have
\begin{equation}
 S = i \sum_{\nu \in \mz} \int_{-1/2}^{1/2} \int_{-1/2}^{1/2} e(-2\alpha \sin(2 \pi t)) e(-2\beta \sin(2 \pi u)) e(t+u)  \widehat{w_{1}}(\nu - 2t-4u - \half) dt du.
\end{equation}
Let $S = S_+ + S_-$ where $S_+$ corresponds to even values of $\nu$, while $S_-$ corresponds to odd values of $\nu$.  In $S_+$ we change variables $t \rightarrow t + \frac{\nu}{2}$ while for $\nu$ odd we do $t \rightarrow t + \frac{\nu-1}{2}$, giving
\begin{equation}
 S_+ = i \intR \int_{-1/2}^{1/2} e(-2\alpha \sin(2 \pi t)) e(-2\beta \sin(2 \pi u)) e(t+u) \widehat{w_{1}}(- 2t-4u - \half) du dt
\end{equation}
and
\begin{equation}
 S_- = i \intR \int_{-1/2}^{1/2} e(-2\alpha \sin(2 \pi t)) e(-2\beta \sin(2 \pi u)) e(t+u) \widehat{w_{1}}(- 2t-4u + \half) du dt.
\end{equation}
Next we reverse the orders of integration and change variables $t \rightarrow t - 2u \mp 1/4$ (depending on if it is $S_+$ or $S_-$), giving (after some brief calculation)
\begin{equation}
 S_{\pm } = \pm \int_{-1/2}^{1/2} e(-2 \beta \sin(2 \pi u)) e(-u) \intR e(\pm 2 \alpha \cos(2 \pi (2u-t))) \widehat{w_{1}}(-2t) e(t) dt du.
\end{equation}
To complete the proof, notice that
 $e(t) \widehat{w_{1}}(-2t) = \half  \widehat{w}(-t)$.
\end{proof}

\begin{mylemma}
\label{lemma:IK}
 Suppose that $\alpha \ll K^{2+\varepsilon}$.  If $\alpha \gg K^{2-\varepsilon}$, we have that $I_{K, \pm \alpha}(u) \ll K^{-1+\varepsilon}$ except possibly if $|\sin(4 \pi u)| \ll K^{-1+\varepsilon}$.  If $\alpha \ll K^{2-\varepsilon}$, then there exists $J$ depending on $\varepsilon$ only and absolute constants $a_j$ such that
\begin{equation}
 I_{K, \pm \alpha}^{T}(u) = e(\pm 2 \alpha \cos(4 \pi u)) \sum_{j \leq J} a_j (\alpha \cos(4 \pi u))^j w^{(2j)}(\mp 4 \pi \alpha  \sin(4 \pi u) ) + O(K^{-1 + \varepsilon}).
\end{equation}
\end{mylemma}
Remark.  In our desired application we only require $\alpha \ll K^{2+\varepsilon}$ but with extra work the analysis of $S(\alpha, \beta)$ can be extended to larger values of $\alpha$.
\begin{proof}
Recall that $w$ is assumed to satisfy \eqref{eq:wprop}, and so by integration by parts we have for each $j=0,1,2, \dots$, and any $A > 0$,
\begin{equation}
\label{eq:whatprop}
\frac{d^j}{dt^j} \widehat{w}(t) \ll_{j,A} K^{1+j} (1 + |t| K)^{-A}.
\end{equation}
 By \eqref{eq:whatprop}, the contribution to $I_{K, \pm \alpha}(u)$ from $|t| \geq K^{-1+\varepsilon}$ is very small.  That is, if we define $I_{K,\alpha}^{T}(u)$ to be the truncated integral with $|t| \leq K^{-1+\varepsilon}$, then $I_{K,\alpha}(u) = I_{K,\alpha}^{T}(u) + O(K^{-100})$.

Write $\cos(x-y) = \cos(x) \cos(y) + \sin(x) \sin(y)$, and take a Taylor expansion so that
\begin{equation}
 I_{K, \pm \alpha}^{T}(u) = \int_{-K^{-1+\varepsilon}}^{K^{-1+\varepsilon}} e \{ \pm 2 \alpha \cos(4 \pi u) [ 1- 2 \pi^2 t^2 + O(t^4)] \pm 2 \alpha \sin(4 \pi u) [2 \pi t + O(t^3)] \} \widehat{w}(-t) dt.
\end{equation}
With shorthand $c = \cos(4 \pi u)$, $s = \sin(4 \pi u)$, and by another Taylor expansion using $\alpha K^{-3} \ll K^{-1+\varepsilon}$, we have
\begin{equation}
 I_{K, \pm \alpha}^{T}(u) = e(\pm 2 \alpha c) \int_{|t| \leq K^{-1+\varepsilon}} e(\pm 4 \pi \alpha ts \mp 4 \pi^2 \alpha t^2  c) \widehat{w}(-t)dt + O(K^{-1 + \varepsilon}).
\end{equation}
If $|\alpha s K^{-1}| \gg K^{2\varepsilon}$ then one can repeatedly integrate by parts to show that this integral is $O(K^{-100})$; it is easiest to combine $e(\mp 4 \pi^2 \alpha t^2 c)$ with $\widehat{w}(-t)$ in this procedure. If $\alpha \gg K^{2-\varepsilon}$ and $|\alpha s K^{-1}| \ll K^{2\varepsilon}$ then $|\sin(4 \pi u)| \ll K^{-1+\varepsilon}$, proving the first statement of the lemma.

Now suppose that $\alpha \ll K^{2-3\varepsilon}$, so $|\alpha t^2| \ll K^{-\varepsilon}$.  In this case, we may expand the quadratic term into Taylor series, giving
\begin{equation}
 I_{K, \pm \alpha}^{T}(u) = e(\pm 2 \alpha c) \sum_{j \leq J} a_j (\alpha c)^j \int_{-K^{-1+\varepsilon}}^{K^{-1+\varepsilon}} (2 \pi i t)^{2j} e( \pm 4 \pi \alpha s t) \widehat{w}(-t) dt + O(K^{-\varepsilon J}) + O(K^{-1 + \varepsilon}).
\end{equation}
Here $a_j$ are certain absolute constants, and in particular $a_0 = 1$.
We pick $J = \lfloor 1/\varepsilon \rfloor + 1$.  Next we extend the range of integration back to $\mr$ without making a new error term.  Using
\begin{equation}
 \intR (2 \pi i t)^{2j} \widehat{w}(-t) e(-\lambda t) dt =  w^{(2j)}(\lambda),
\end{equation}
we have
\begin{equation}
 I_{K, \pm \alpha}^{T}(u) = e(\pm 2 \alpha \cos(4 \pi u)) \sum_{j \leq J} a_j (\alpha c)^j w^{(2j)}(\mp 4 \pi \alpha  s ) + O(K^{-1 + \varepsilon}). \qedhere
\end{equation}
\end{proof}

Now we quote some results given in Huxley's book [Hu].
\begin{mylemma}[[Hu{]} Lemma 5.1.2]
\label{lemma:HuxleyFirst} 
Let $f(x)$ be real and differentiable on the interval $A < x < B$ with $f'(x)$ monotone and $f'(x) \geq \kappa > 0$ on $(A,B)$.  Let $g(x)$ be real, and let $V$ be the total variation of $g$ on the closed interval $[A,B]$ plus the maximum modulus of $g(x)$ on $[A,B]$.  Then
\begin{equation}
 |\int_{A}^{B} g(x) e(f(x)) dx | \leq \frac{V}{\pi \kappa}.
\end{equation}
\end{mylemma}
\begin{mylemma}[[Hu{]} Lemma 5.5.6]
\label{lemma:HuxleyStationary}
 Let $f(x)$ be a real function, four times continuously differentiable for $A \leq x \leq B$, and let $g(x)$ be a real function, three times continuously differentiable for $A \leq x \leq B$.  Suppose that there are positive parameters $M, N, T, U$ with 
\begin{equation}
M \geq B-A, \qquad N\geq M T^{-1/2}
\end{equation}
and positive constants $C_r$ such that for $A \leq x \leq B$,
\begin{equation}
\Big|f^{(r)}(x)\Big| \leq C_r T M^{-r}, \qquad |g^{(s)}(x)| \leq C_s U N^{-s}
\end{equation}
for $r=2,3, 4$, and $s=0,1,2, 3$, and
\begin{equation}
 f''(x) \geq \frac{T}{C_2 M^2}.
\end{equation}
Suppose also that $f'(x)$ changes sign from negative to positive at a point $x=x_0$ with $A < x_0 < B$.  If $T$ is sufficiently large in terms of the constants $C_r$, then we have
\begin{multline}
 \int_{A}^{B} g(x) e(f(x)) dx = \frac{g(x_0) e(f(x_0) + \frac18)}{\sqrt{f''(x_0)}} + O(\frac{|g(B)|}{|f'(B)|} + \frac{|g(A)|}{|f'(A)|}) 
\\
+ O(\frac{M^4 U}{T^2}(1 + \frac{M}{N})^2 ( \frac{1}{(x_0-A)^3} + \frac{1}{(B-x_0)^3})) + O(\frac{MU}{T^{3/2}} (1 + \frac{M}{N})^2).
\end{multline}
\end{mylemma}

\begin{proof}[Proof of Theorem \ref{thm:Salphabeta}]
 Applying Lemma \ref{lemma:IK} to $S(\alpha, \beta)$ given in the form of \eqref{eq:SIntegralForm}, we have
\begin{equation}
\label{eq:Sfg}
 S(\alpha, \beta) =  \sum_{\pm} \int_{-1/2}^{1/2} e(f(u)) g(u) du + O(K^{-1+\varepsilon}),
\end{equation}
where
\begin{equation}
 f(u) = \pm 2 \alpha \cos(4 \pi u) - 2 \beta \sin(2 \pi u), 
\end{equation}
and
\begin{equation}
\label{eq:gformula}
 g(u) = \pm  \half e(-u) \sum_{j \leq J} a_j (\alpha \cos(4 \pi u))^j  w^{(2j)}(\mp 4 \pi \alpha \sin(4 \pi u)).
\end{equation}
Technically, we need to assume $\alpha \ll K^{2-\varepsilon}$ to appeal to Lemma \ref{lemma:IK}, and we henceforth assume this since if $\alpha \gg K^{2-\varepsilon}$ then the main term in Theorem \ref{thm:Salphabeta} is absorbed by the error term. 

We find it convenient to fix the $\pm$ sign to be $-$; the $+$ case can be reduced to this form by conjugating the integral and changing variables $u \rightarrow -u$.  Observe that conjugating \eqref{eq:Sasymp} effectively swaps the two main terms.

We compute 
\begin{equation}
\label{eq:f'}
  f'(u) = 8 \pi  \alpha \sin(4 \pi u)- 4 \pi \beta \cos(2 \pi u)). 
\end{equation}
The support on $g$ means that $4 \pi \alpha \sin(4 \pi u) \geq K$ whence $f'(u) \geq 2K - 4 \pi \beta \cos(2 \pi u)$.  Since $\cos(2 \pi u) < 0$ outside of $[-1/4, 1/4]$, we have $f'(u) \geq 2K$ on the intervals $[-1/2, -1/4]$, and $[1/4, 1/2]$ and Lemma \ref{lemma:HuxleyFirst} shows that the contribution from this stretch is $O(K^{-1})$.  We can further restrict to the interval $[0, 1/4]$ since $g(u)$ vanishes on $[-1/4, 0]$.  Next we make a small simplification on $g$.  If $\alpha \ll K^{1+\varepsilon}$ then the terms with $j \geq 1$ in \eqref{eq:gformula} are $O(\alpha K^{-2}) = O(K^{-1+\varepsilon})$.  If $\alpha \gg K^{1+\varepsilon}$ then the support on $w$ means that $g(u) = 0$ unless $u \asymp K/\alpha$ and hence $\cos(4 \pi u) = 1 + O(\frac{K^2}{\alpha^2})$.  Gathering cases, we find that for $0 \leq u \leq \frac14$,
\begin{equation}
 g(u) = e(-u)  W(4 \pi \sin(4 \pi u)) + O(K^{-1+\varepsilon}),
\end{equation}
where $W$ is the following function which clearly satisfies \eqref{eq:wprop}:
\begin{equation}
 W(x) = - \half \sum_{j \leq J} a_j \alpha^j w^{(2j)}(x).
\end{equation}

Notice that $\sin(2 \pi u)$ is nonnegative and one-to-one on $[0, 1/4]$, and so is $\cos(2 \pi u)$. Now we change variables $v = \sin(2 \pi u)$.  Then $dv = 2 \pi \cos(2 \pi u) du = 2 \pi \sqrt{1-v^2} du$, $\sin(4 \pi u) = 2 \sin(2 \pi u) \cos(2 \pi u) = 2 v \sqrt{1-v^2}$, $\cos(4 \pi u) = 1 - 2 v^2$, and $e(-u) = \sqrt{1-v^2} - i v$.  

Then the contribution to \eqref{eq:Sfg} from the $-$ sign case is
\begin{equation}
\int_0^{1} e(-2 \alpha (1-2 v^2)  - 2 \beta v)   W(8 \pi \alpha v \sqrt{1-v^2}) \frac{-\sqrt{1-v^2} + iv}{4 \pi \sqrt{1-v^2}} dv + O(K^{-1+\varepsilon}).
\end{equation}
Let $h(x) = K x^{-1} W(x)$, which satisfies \eqref{eq:wprop}, so that the weight function above becomes
\begin{equation}
c_1 W(8 \pi \alpha v\sqrt{1-v^2}) + c_2 \frac{\alpha v}{K} v h(8 \pi \alpha v \sqrt{1-v^2}) =: r(v),
\end{equation}
where $c_1$ and $c_2$ are certain absolute constants.
Observe that $v \leq 1/2$ is in the support of $r$ only if $v \asymp K/\alpha$, and $v \geq 1/2$ is in the support of $r$ only if $1-v \asymp \frac{K^2}{\alpha^2}$.  Therefore if $\alpha \geq K^{1+\varepsilon}$ then $r$ is supported only on a short interval near $0$ of length $K/\alpha$, and on a short interval near $1$ of length $K^2/\alpha^2$.  As $\alpha$ approaches $K$ these intervals may merge into one longer interval.
Some careful thought shows that the derivatives of $r$ satisfy
\begin{equation}
r^{(j)}(v) \ll 
\begin{cases}
(\alpha/K)^{j}, \qquad &\text{if } v \leq 1/2, \\
(\alpha/K)^{j+1}, \qquad &\text{if }  v \geq 1/2,
\end{cases}
\end{equation}
and also that $\int_0^{1} |r'(v)| dv \ll \frac{\alpha}{K}$ (do not forget the support of $r$).
Now this integral takes the form
\begin{equation}
e(- 2 \alpha) \int_0^{1} e(f_1(v)) r(v) dv, \qquad f_1(v) = 4 \alpha v^2 - 2 \beta v.
\end{equation}
Recalling $\gamma = \frac{\beta}{4 \alpha}$, we have
\begin{equation}
f_1'(v) = 8\alpha ( v - \gamma), \qquad f_1''(v) = 8 \alpha.
\end{equation}
First we note that if $\gamma \geq 2$ then $f_1'(v) \leq -8 \alpha$ and so Lemma \ref{lemma:HuxleyFirst} gives that the integral is $O(K^{-1})$, consistent with Theorem \ref{thm:Salphabeta}.

Now suppose $\gamma < 2$.
Since $r(v)$ is identially zero near $1$, we can freely extend (for convenience) the integral to cover $[-1,3]$ with the convention that $r(v)$ is zero outside of $[1,2]$.  Now we can apply Lemma \ref{lemma:HuxleyStationary} to this integral.  Actually it is important to consider the cases $\gamma \leq 1/4$ and $\gamma > 1/4$ separately.  First suppose $\gamma > 1/4$ (which implies $\alpha \ll K^{1+\varepsilon}$ by the definition of $\gamma$ and the bound $\beta \ll K^{1+\varepsilon}$).  In this case, we have in
Huxley's notation, $T = \alpha$, $M =3$, $U=\alpha/K$, $N = K/\alpha$, and hence the error term in stationary phase is
\begin{equation}
\ll \frac{\alpha}{K \alpha^2} (1 + \frac{\alpha}{K})^2 + \frac{\alpha}{K \alpha^{3/2}} (1 + \frac{\alpha}{K})^2 \ll K^{-3/2 + \varepsilon}.
\end{equation}
Now suppose $\gamma \leq 1/4$.  On the interval $v \geq 1/2$ we have the bound $f_1'(v) \gg \alpha$ and so as in the previous paragraph the contribution from this part is $O(K^{-1})$.  On $v \leq 1/2$, we may apply Lemma \ref{lemma:HuxleyStationary} as before but now the difference is that $U = 1$ instead of $\alpha/K$.  Then the error term becomes
\begin{equation}
\ll \frac{1}{\alpha^2} (1 + \frac{\alpha}{K})^2 + \frac{1}{\alpha^{3/2}} ( 1 + \frac{\alpha}{K})^2 \ll \frac{\alpha^{1/2} }{ K^{2}},
\end{equation}
which is $O(K^{-1+\varepsilon})$ using $\alpha \ll K^{2+\varepsilon}$.

All that remains is to check that the main term in Lemma \ref{lemma:HuxleyStationary} is consistent with Theorem \ref{thm:Salphabeta}.  We have $f_1(\gamma) = -\frac{\beta^2}{4 \alpha}$, $f_1''(\gamma) = 8\alpha$, and hence this main term is
\begin{equation}
e(-2 \alpha - \frac{\beta^2}{4 \alpha} + \frac18) (8\alpha)^{-1/2} [c_1 W(2 \pi \beta  \sqrt{1-\gamma^2}) + c_2 \frac{\beta}{4K} \gamma h(2 \pi \beta  \sqrt{1-\gamma^2})],
\end{equation}
which is of the form desired for \eqref{eq:Sasymp}, with
\begin{equation}
\label{eq:H-formula}
H_{-}(x,y) = 8^{-1/2} e(\tfrac18) (c_1 W(2 \pi x) + c_2 \frac{\beta}{4K} y h(2 \pi x)). \qedhere
\end{equation}
\end{proof}

\section{Proof of Theorem \ref{thm:Poisson}}
\label{section:PoissonProof}
For the error term $T_E$, we have using the Weil bound for Kloosterman sum of modulus $c_1$ and the trivial bound for the Kloosterman sum of modulus $c_2$, that
\begin{equation} 
T_E \ll K^{-1+\varepsilon} c_2 c_1^{1/2+\varepsilon} N,
\end{equation}
as desired.
The main terms $T_{\pm}$ take the form
\begin{equation}
 T_{\pm} = e(\pm \frac{r_1 c_1}{4 c_2^2 r_2})\sum_n S(n^2, r_2^2;c) S(r_1 n, 1;d) \frac{H_{\pm}(\frac{\sqrt{n r_1}}{c_2}\sqrt{1-\frac{r_1 c_1^2}{16r_2^2 c_2^2 n}}, \frac{\sqrt{r_1}c_1}{4 r_2 c_2 \sqrt{n}})}{\sqrt{nr_2/c}} e(\pm \frac{2nr_2}{c_1}).
\end{equation}
This is an abuse of notation because in the development of $H_{\pm}$ we neglected the dependence of the weight function on $n$.  However, recalling the development \eqref{eq:H-formula} we obtained $H_{\pm}$ very explicitly in terms of $W(n,k)$ (with $n$ suppressed in the notation).  A little thought shows that it is sufficient to bound a sum of the form
\begin{equation}
T_{\pm}' = e(\pm \frac{r_1 c_1}{4 c_2^2 r_2})\sum_n S(n^2, r_2^2;c_1) S(r_1 n, 1;c_2) \frac{H_{\pm}(n, \frac{\sqrt{n r_1}}{c_2}\sqrt{1-\frac{r_1 c_1^2}{16r_2^2 c_2^2 n}}, \frac{\sqrt{r_1}c_1}{4 r_2 c_2 \sqrt{n}})}{\sqrt{nr_2/c_1}} e(\pm \frac{2nr_2}{c_1}),
\end{equation}
where $H_{\pm}(t, x, y)$ has support in $t \asymp N$, $x \asymp K$, and $y \leq 1$, and satisfies $H_{\pm}^{(i,j,k)}(t,x,y) \ll N^{-i} K^{-j}$ for any $i,j,k \in \{ 0, 1, 2, \dots \}$.  At this point we forsake any cancellation except in $n$ and therefore suppress the dependence of $H_{\pm}$ on other variables besides $n$.  That is, we write
\begin{equation} \label{eq:Tpmdef}
 T_{\pm}' = e(\pm \frac{r_1 c_1}{4 c_2^2 r_2}) \sqrt{\frac{c_1}{Nr_2}} \sum_n U(n) S(n^2, r_2^2;c_1) S(r_1 n, 1;c_2) e(\pm \frac{2nr_2}{c_1}),
\end{equation}
where $U(n)$ is a certain smooth function satisfying
\begin{equation}
\label{eq:Uderiv}
 U^{(j)}(x) \ll_{j, A} (\frac{K^{\varepsilon}}{N})^j (1 + \frac{x}{N})^{-A}.
\end{equation}
We get this using 
\begin{equation}
 \frac{r_1 N}{K^2 c_2^2} \ll K^{\varepsilon}, \qquad \frac{\sqrt{r_1} c_1}{\sqrt{N} r_2} \ll 1,
\end{equation}
the latter following from $H_{\pm}(t,x, y) = 0$ for $y > 1$.

\begin{mylemma} \label{Tpm}
 Let $U(x)$ be any function compactly-supported on the positive reals satisfying \eqref{eq:Uderiv}.  Then
$T_{\pm}'$ defined by \eqref{eq:Tpmdef} with $c_2 \ll K^{\varepsilon}, c_1 \ll K^{1+\varepsilon}$ satisfies
\begin{equation}
 T_{\pm}' \ll \sqrt{N} c_2^{1/2 + \varepsilon} (c_0 c')^{1+\varepsilon} \delta_{c_0=\square} + O(K^{-100}).
\end{equation}
\end{mylemma}
\begin{proof}
 Breaking the sum over $n$ into arithmetic progressions and applying Poisson summation, we have
\begin{equation}
\label{eq:TpmPoisson}
 |T_{\pm}'| = \sqrt{\frac{c_1}{Nr_2}} \Big|\sum_{a \shortmod{c_1 c_2}} S(a^2, r_2^2;c_1) S(r_1 a, 1;c_2) e(\frac{\pm 2 a r_2}{c_1}) \frac{1}{c_1 c_2} \sum_{\nu \in \mz} e(\frac{a \nu}{c_1 c_2}) \widehat{U}(\frac{\nu}{c_1 c_2})\Big|.
\end{equation}
The usual integration by parts argument shows that
\begin{equation}
 \widehat{U}(y) \ll N (1 + \frac{|y| N}{K^{\varepsilon}})^{-A}.
\end{equation}
Therefore the truncation condition on $\nu$ in \eqref{eq:TpmPoisson} is that $|\nu| \ll K^{\varepsilon} c_1 c_2/N$.
Now in our application we may assume $\alpha, \beta \gg K$ where $\alpha \asymp \frac{N r_2}{c_1}$ and $\beta \asymp \frac{\sqrt{N r_1}}{c_2}$.  Using the lower bounds on $\alpha$ and $\beta$, we see that $c_1 \ll \frac{Nr_2}{K}$ and $c_2 \ll \frac{\sqrt{N r_1}}{K}$.  Recalling $N r_1 r_2^2 \ll K^{2+\varepsilon}$, we see that the practical truncation condition on $\nu$ is that $|\nu| \ll K^{-1+\varepsilon}$, meaning that the only relevant term is $\nu = 0$ (all the others can be bounded by an arbitrarily small power of $K$).  That is, we have the bound
\begin{equation}
 |T_{\pm}'|  \ll \frac{\sqrt{N}}{\sqrt{c_1}c_2} \Big|\sum_{a \shortmod{c_1 c_2}} S(a^2, r_2^2;c_1) S(r_1 a, 1;c_2) e(\frac{\pm 2 a r_2}{c_1}) \Big| + O(K^{-100}).
\end{equation}
Next we write $c_1=c_0 c'$ where $(c_0,c')=1$ and $c'| c_2 ^{\infty}$. Then using the Chinese reminder theorem we have
$a \equiv a_0 c' c_2 \overline{c_2 c'} + a' c_0 \overline{c_0}$ where $a_0$ runs modulo $c_0$, and $a'$ runs modulo $c' c_2$. 
Here $c_0 \overline{c_0} \equiv 1 \pmod{c' c_2}$, and $c' \overline{c'} \equiv 1 \equiv c_2 \overline{c_2} \pmod{c_0}.$ 
We use the relation $$S(m,n;cq)=S(\overline{q}m, \overline{q}n;c)S(\overline{c}m, \overline{c}n;q) \quad \text{for} \quad (c,q)=1$$
to separate the sum over $a_0$ modulo $c_0$ and $a'$ modulo $c' c_2$.  That is, we have
\begin{multline}
|T_{\pm}'|  \ll \frac{\sqrt{N}}{\sqrt{c_1}c_2} \Big|\sum_{a_0 \shortmod{c_0}}  S(a_0^2 \overline{c'^2}, r_2^2;c_0) e(\frac{\pm 2 a_0 r_2 \overline{c'}}{c_0}) \Big|
\\
\times \Big|\sum_{a' \shortmod{c' c_2}} S(a'^2 \overline{c_0^2}, r_2^2;c') S(r_1 a', 1;c_2) e(\frac{\pm 2 a' r_2 \overline{c_0}}{c'}) \Big| + O(K^{-100})
\end{multline}
This inner sum over $a'$ is bounded with the Weil bound by
\begin{equation}
|\sum_{a'}| \ll (c_2 c')^{1/2 + \varepsilon} \sum_{a' \shortmod{c' c_2}} (a', c') \ll (c_2 c')^{3/2 + \varepsilon}.
\end{equation}
Thus
\begin{equation}
 |T_{\pm}'|  \ll \frac{\sqrt{N}}{\sqrt{c_0}} c'^{1+ \varepsilon} c_2^{1/2+\varepsilon} \Big|\sum_{a_0 \shortmod{c_0}}  S(a_0^2 \overline{c'^2}, r_2^2;c_0) e(\frac{\pm 2 a_0 r_2 \overline{c'}}{c_0}) \Big| + O(K^{-100}).
\end{equation}
Then we note
\begin{equation}
\label{eq:asum}
\mathcal{T}_{r_2}(c_0) := \sum_{a_0 \shortmod{c_0}}  S(a_0^2 \overline{c'^2}, r_2^2;c_0) e(\frac{\pm 2 a_0 r_2 \overline{c'}}{c_0})
= \sumstar_{h \shortmod{c_0}} \sum_{a_0 \shortmod{c_0}} e(\frac{h (a_0 \overline{c'} \pm \overline{h} r_2)^2}{c_0}).
\end{equation}
Changing variables $a_0 \rightarrow c'(a_0 \mp \overline{h} r_2)$ shows that $\mathcal{T}_{r_2}(c)$ is independent of $r_2$, and in fact
\begin{equation}
\label{eq:Tcomp}
\mathcal{T}_{r_2}(c_0)= \sumstar_{h \shortmod{c_0}} \sum_{a \shortmod{c_0}} e(\frac{h a^2}{c_0}).
\end{equation}
We claim that $\mathcal{T}(c_0) :=\mathcal{T}_{r_2}(c_0)$ is zero unless $c_0$ is a square, in which case it is $\phi(c_0) \sqrt{c_0}$.  
An easy argument with the Chinese remainder theorem shows that $\mathcal{T}(c_0)$ is multiplicative in terms of $c_0$, so it suffices to check the formula for $c_0$ a prime power.  
The outer sum over $h$ becomes a Ramanujan sum which we evaluate in terms of a divisor sum.  Then \eqref{eq:asum} becomes
\begin{equation}
 \sum_{a \shortmod{c_0}} \sum_{\substack{b|c_0 \\ b|a^2}} b \mu(c_0/b) = c_0 \sum_{b|c_0} \mu(c_0/b) \sum_{\substack{a \shortmod{b} \\ a^2 \equiv 0 \shortmod{b}}} 1.
\end{equation}
If $c_0 = p^j$ with $j$ odd then $\mathcal{T}(p^j) = 0$ since the inner sum over $a$ takes the value $p^{(j-1)/2}$ for both values $b=p^j$ and $b=p^{j-1}$.  
If $c_0=p^{j}$ with $j$ even then an easy calculation shows $\mathcal{T}(p^{j}) = c_0(p^{\frac{j}{2}} - p^{\frac{j}{2}-1}) = \phi(c_0) \sqrt{c_0}$, as desired.
\end{proof}

\end{document}